\documentclass[12pt]{amsart}
\usepackage{url,amstext,amsfonts,amssymb,amscd,amsbsy,amsmath,amsthm,verbatim,mathrsfs}
\usepackage{ifthen, fullpage}
\usepackage{color}
\usepackage{amsthm}
\usepackage{latexsym}
\usepackage[all]{xy}
\usepackage{enumerate}
\usepackage{cite} % for bibliography
\usepackage[enableskew]{youngtab} % For young tableaux
\usepackage{hyperref}

\newtheorem{lemma}{Lemma}[section]
\newtheorem{prop}[lemma]{Proposition}
\newtheorem{cor}[lemma]{Corollary}
\newtheorem{thm}[lemma]{Theorem}
\newtheorem{defn}[lemma]{Definition}

\theoremstyle{remark}
\newtheorem{remark}[lemma]{Remark}
\newtheorem{example}[lemma]{Example}

\newcommand{\defi}[1]{\textsf{#1}} % for defined terms
\newcommand{\isom}{\cong}

\newcommand{\Id}{\text{Id}}

\newcommand{\Hom}{\operatorname{Hom}}
\newcommand{\Gr}{\operatorname{Gr}}

\newcommand{\GL}{\operatorname{GL}}

\newcommand{\Osh}{{\mathcal O}}

\newcommand{\PP}{\mathbb{P}}
\newcommand{\CC}{\mathbb{C}}

\newcommand{\Seg}{\operatorname{Seg}}
\newcommand{\Sub}{\operatorname{Sub}}
\def\bw#1{{\textstyle\bigwedge^{\hspace{-.2em}#1}}}

\title{Secant varieties of $\PP^2 \times \PP^n$ embedded by $\Osh(1,2)$}
\author{Dustin Cartwright}
\address{Department of Mathematics\\
University of California\\
Berkeley, CA 94720, USA}
\email{dustin@math.berkeley.edu}
\author{Daniel Erman}
\address{Department of Mathematics\\
Stanford University\\
Stanford, CA 94305, USA}
\email{derman@math.stanford.edu}
\author{Luke Oeding}
\address{Dipartimento di Matematica ``U. Dini'' \\
Universit\'a degli Studi di Firenze \\
Firenze, Italy}
\email{oeding@math.unifi.it}
\begin{document}

\begin{abstract}
We describe the defining ideal of the $r$th secant variety of $\PP^2\times \PP^n$ embedded by $\mathcal O(1,2)$, for arbitrary $n$ and $r\leq 5$.
We also present the Schur module decomposition of the space of generators of each such ideal.
%This thus provides a new family where we can explicitly compute the border ranks of partially symmetric $3$-tensors.
Our main results are based on a more general construction for producing explicit matrix equations that vanish on secant varieties of products of projective spaces.  This extends previous work of Strassen and Ottaviani.
%
%We investigate the border ranks of $3$-tensors or, equivalently, the secant varieties of certain Segre-Veronese embeddings of projective spaces.  We define explicit matrix equations which vanish on these secant varieties. We also present the decomposition of these equations into irreducible Schur modules.  In the case of $\PP^2\times \PP^n$ embedded by $\mathcal O(1,2)$, and for $r \leq 5$, we prove that our matrix equations generate the entire defining ideal for the $r$th secant variety, for $r\leq 5$.  This extends previous work of Strassen and Ottaviani.
\end{abstract}
\maketitle

\section{Introduction}\label{sec:intro}
Let $U$, $V$, and~$W$ be complex vector spaces of dimension $m$, $n$, and~$k$
respectively, and let
$x$ be an element in the tensor product of of their duals, $U^* \otimes V^*
\otimes W^*$.
%\footnote{We choose $x$ to be in
%the vector space duals so that the equations on~$x$ can be written in terms of
%$U$, $V$, and~$W$ themselves.}
The \defi{border rank} of $x$ is the minimal
$r$ such that the corresponding point
$[x] \in \PP(U^* \otimes V^* \otimes W^*)$ lies in the $r$th secant variety of
the Segre variety of $\PP(U^*) \times \PP(V^*) \times \PP(W^*)$. 
Similarly, for a symmetric tensor $x \in S^3U^*$ or a
partially symmetric tensor $x \in U^* \otimes S^2 V^*$, the \defi{symmetric border rank}
and the \defi{partially symmetric border rank} are the smallest $r$ such that $[x]$ is in
the $r$th secant variety of the Veronese or the Segre-Veronese variety,
respectively.  Developing effective techniques for computing the border rank of tensors
 is an active area of research which spans classical algebraic geometry and 
 representation theory~\cite{lw,lm,lm2,landsberg-2009,LanOtt,ottaviani}.

In the partially symmetric case, the secant varieties of $\PP^1 \times \PP^{n-1}$
embedded by $\Osh(1,2)$ are closely related to standard results about pencils of
symmetric matrices.  Moreover, the non-symmetric analogue is $\PP^1 \times
\PP^{n-1} \times \PP^{k-1}$ embedded by $\Osh(1,1,1)$, and the defining
equations of all of its secant varieties are known by work of Landsberg and
Weyman~\cite[Thm.~1.1]{lw}.  We record the partially symmetric analogue in
Proposition~\ref{prop:m-equals-1}.

Our main result is Theorem~\ref{thm:main}, which focuses on the next case:
secant varieties of $\PP^2\times \PP^{n-1}$ embedded by $\Osh(1,2)$. We
give two explicit matrices, and we prove that, when $r\leq 5$, their minors
and Pfaffians, respectively, generate the defining ideal for these secant
varieties.
To illustrate, fix a basis $\{e_1,e_2,e_3\}$ of~$U^*$.  
We may then express any point $x\in \PP(U^*\otimes S^2 V^*)$ as $x=e_1\otimes
A_1+e_2\otimes A_2+e_3\otimes A_3$ where each $A_i\in S^2 V^*$ can be
represented by an $n\times n$ 
symmetric matrix. With the ordered triplet of matrices $(A_1, A_2,A_3)$ serving
as coordinates on $\PP(U^*\otimes S^2 V^*)$, our main result is the following,
which is a restatement of Theorem~\ref{thm:main}:
\begin{thm}\label{thm:intro}
Let $Y$ be the image of $\PP^2\times \PP^{n-1}$ in
$\PP^{3\binom{n+1}{2}-1} \cong \PP(U^* \otimes S^2 V^*)$ embedded by
$\Osh(1,2)$.  For any $r\leq 5$, the 
$r$th secant variety of~$Y$ is defined by the prime ideal generated by the $(r+1)\times (r+1)$ minors of the $n\times 3n$ block matrix
\begin{equation}\label{eq:flattening}
\begin{pmatrix}
A_1 & A_2 & A_3
\end{pmatrix}
\end{equation}
and by the $(2r+2)\times (2r+2)$ principal Pfaffians of the $3n\times 3n$ block matrix
\begin{equation}\label{eq:skew-sym}
\begin{pmatrix}
0&A_3&-A_2\\ -A_3 & 0 & A_1 \\
A_2&-A_1&0\end{pmatrix}.
\end{equation}
\end{thm}

The matrices which appear in the statement of the above theorem are
examples of what we call the ``exterior flattenings'' of a $3$-tensor (see \S\ref{sec:kappa}),
and the construction of these matrices is motivated by the $\kappa$-invariant of a $3$-tensor,
as introduced in~\cite[\S1.1]{ev}.  The minors of the exterior flattenings of a $3$-tensor
impose necessary equations on a wide array of secant varieties of 
Segre-Veronese embeddings of products of projective spaces.  
The minors of these exterior flattenings simultaneously generalize
both the minors obtained from flattenings of a $3$-tensor and the determinantal equations of
\cite[Lem.~4.4]{strassen} and \cite[Thm.~3.2]{Ottaviani07}.

Under the hypotheses of Theorem~\ref{thm:intro}, the minors and Pfaffians of these exterior flattenings
are insufficient to generate the ideal of the $r$th secant variety for $r \geq 7$.  In other words,
Theorem~\ref{thm:intro} if false is $r\geq 7$, and we do not know if 
Theorem~\ref{thm:intro} holds when $r=6$.
See Example~\ref{ex:notequal} for more details. Note that by
\cite[Cor.~1.4(ii)]{abo}, these secant varieties have the expected dimension
except when $n$ is odd and $r = n + (n+1)/2$.

The proof of our main result uses a mix of representation theory
and geometric techniques for studying determinantal varieties.  
We first introduce the relevant determinantal ideals and we use their
equivariance properties to relate these ideals as the size of the
tensor varies.
Next, we apply this relation in order to understand the defining ideals of certain auxiliary varieties
known as the subspace varieties $\Sub_{m',n'}$ (see Definition~\ref{defn:subspace}).
We then prove our main result in the special case that $n=r$, by
relating the secant variety with the 
variety of commuting symmetric $n\times n$ matrices.  A similar
idea has appeared in several instances previously~\cite{strassen, Ottaviani07,shuchao}.  
This step requires $r\leq 5$.
Finally, we prove our main result by blending our results about subspace varieties 
with our knowledge about the case $n=r$.

Partially symmetric $3$-tensors are closely related to
the study of vector spaces of quadrics, which arise naturally in algebraic geometry.
For instance, in the study of Hilbert schemes of points, 
border rank is connected to the smoothability of zero
dimensional schemes~\cite{cevv, ev}.
%In addition, imagine a subvariety $Z \subseteq \mathbb PV^*$ defined by the
%vanishing of $(m+1)$ quadric polynomials $f_0, \dots, f_m$.  The $(m+1)$-tuple
%$\mathbf{f}=(f_0, \dots, f_m)$ defines a point of $\mathbb P(\CC^{m+1}\otimes
%S^2 V^*)$.  By stratifying $\mathbb P(\CC^{m+1} \otimes S^2 V^*)$ by border
%rank, we obtain invariants of $\mathbf{f}$ which
%may reveal interesting geometric information about $Z$.  
%This is essentially
%what happens with the Hilbert schemes of points, as~\cite[Thm.\ 1.2,
%Prop.\ 4.3]{ev} illustrate a connection between border rank and smoothability.
As another example, ~\cite[Prop.~6.3]{Ottaviani07} relates the border rank of a partially symmetric tensor 
$x\in \mathbb C^3\otimes S^2(\mathbb C^n)$
with properties of the corresponding degree $n$ determinantal curve in $\mathbb P^2$.

Questions about the border rank of partially symmetric tensors also arise in
algebraic statistics~\cite[\S7]{gss}.  For instance, the situation of the
Theorem~\ref{thm:main} corresponds to a mixture of random processes, each
independently sampling from a distribution with $3$ states and sampling twice
from a distribution with $n$ states. The border rank of the observed
distribution corresponds to the number of processes in the mixture.

In signal processing,
a partially symmetric tensor in $U^*\otimes S^{2}V^*$ can be constructed as
the second derivative of the cumulant generating function taken at
$m$~points~\cite{Yeredor2010}.
The matrix equations in Theorem~\ref{thm:main} can be used to study small border
ranks of such tensors in the case $m =3$.

The defining ideal of the $r$th secant variety of $\PP^2\times \PP^{n-1}$ was
previously known in the case when this secant variety is a hypersurface.  This
occurs when $n \geq 4$ is even, and $r=\frac{3n-2}{2}$, and this result follows
from an analogue to Strassen's argument~\cite[\S~4]{strassen}, as shown by
Ottaviani in the remark following Theorem~4.1 in~\cite{Ottaviani07}.  
For historical interest, we note that the hypersurface case $n=4$ and $r=5$
dates to Emil Toeplitz~\cite{toeplitz}.

Theorem~\ref{thm:main} thus provides a new family 
of examples where we can effectively compute the border rank of a partially symmetric tensor.  
Our main result also provides evidence for a partially symmetric analogue of Comon's
Conjecture, which posits that the symmetric rank of a tensor equals the
rank~\cite[\S 5]{cglm}, as discussed in Remark~\ref{rmk:Comon} below.

%
%Our main result is Theorem~\ref{thm:main}, which gives the defining ideal for
%the $r$th secant variety of $\PP U^*\times \PP V^*$ embedded by $\mathcal O(1,2)$
%in the cases $\dim \PP U^*=2$ and $r\leq 5$. We thus provide a new family of examples where we can effectively compute the border rank of a partially symmetric tensor.  

%For instance, let $x\in U^*\otimes S^2 V^*$ with $\dim V^*=4$, and let
%$\{e_1,e_2,e_3\}$ be a basis of $U^*$.  Then we may write $x=e_1\otimes
%A_1+e_2\otimes A_2+e_3\otimes A_3$ with $A_i\in S^2 V^*$ a $4\times
%4$-symmetric matrix.  Theorem~\ref{thm:main} implies that the border rank of $x$
%is the larger of
%\begin{equation*}
%\rank\begin{pmatrix}A_1&A_2&A_3\end{pmatrix} \qquad \text{ and } \qquad
%\frac{1}{2} \rank \begin{pmatrix}
%0&A_3&-A_2\\ -A_3 & 0 & A_1 \\
%A_2&-A_1&0\end{pmatrix}.
%\end{equation*}
%This example is explored in greater detail in Example~\ref{ex:23}

This paper is organized as follows.  In \S\ref{sec:kappa}, we define a vector
$\kappa$ as an invariant of a $3$-tensor.  We use this $\kappa$-invariant to
produce explicit matrix equations which vanish on the secant varieties of
Segre-Veronese embeddings of projective spaces.  To provide a more invariant
perspective, and to connect with previous literature~\cite{lm,lm2,lw,lw09}, we
also provide Schur module decompositions for our matrix equations.  In
\S\ref{sec:kappa-semi}, we restrict to the case of the $\kappa$-invariant of a
partially symmetric tensor.  Here we also provide Schur module decompositions in the
partially symmetric case.  In \S\ref{sec:subspace}, we show that the $\kappa_0$
equations define subspace varieties. We prove our main result,
Theorem~\ref{thm:main}, in \S\ref{sec:m-3}. 

\begin{remark}\label{rmk:characteristic}
Our results giving equations vanishing on the Segre and Segre-Veronese varieties
(Propositions~\ref{prop:vanish} and~\ref{prop:vanish:semi}) hold in arbitrary
characteristic.
However, our proof of Theorem~\ref{thm:main} does not extend to arbitrary
characteristic because it relies on Lemmas \ref{lem:subspace-projection}
and~\ref{lem:dim-P} and~\cite[Thm.~3.1]{bpv}, all of which require
characteristic~$0$.
\end{remark}

\subsection*{Acknowledgments}
We thank Bernd Sturmfels for inspiring our work on this project, and we thank
David Eisenbud, J. M. Landsberg, Chris Manon, Giorgio Ottaviani, Mauricio
Velasco, and Jerzy Weyman for many useful conversations.  We also thank the
makers of Macaulay2~\cite{M2}, Daniel Grayson and Michael Stillman, and the
makers of LiE~\cite{lie}, M.A.A. van Leeuwen, A.M. Coehn, and B. Lisser, as
these two programs were extremely helpful at many stages of research.

The first author was partially supported by the U.S. National Science
Foundation (DMS-0968882). The second author was partially supported by an NDSEG fellowship and by an NSF Research Postdoctoral Fellowship during his work on this project.
The third author was supported by NSF Grant 0853000: International Research Fellowship Program (IRFP)

%%%%%%%%%%%%%%%%%%%%%
%%%%%%%%%%%%%%%%%%%%%
\section{The \texorpdfstring{$\kappa$}{k}-invariant of a $3$-tensor}\label{sec:kappa}
%%%%%%%%%%%%%%%%%%%%%
%%%%%%%%%%%%%%%%%%%%%

From a tensor in $U^* \otimes V^* \otimes W^*$, we will construct a series of
linear maps, whose ranks we define to be the \defi{$\kappa$-invariants} of the tensor. The
$\kappa$-invariants give inequalities on the rank of the tensor, and thus,
determinantal equations which vanish on the secant variety.

There is a natural map $U^* \otimes \bw j U^* \rightarrow
\bw {j+1} U^*$ defined by sending $u \otimes u' \mapsto u \wedge u'$ for
any $0 \leq j \leq m-1$.
This induces an inclusion $U^*\subseteq
\bw j U\otimes \bw {j+1} U^*$. 
By tensoring on both sides by $V^* \otimes W^*$, we get an
inclusion
$U^*\otimes V^*\otimes W^*\subseteq (V\otimes \bw j U^*)^{*}\otimes
(W^*\otimes \bw {j+1} U^*)$. An element of the tensor product on the right-hand side
may be interpreted as a linear homomorphism, meaning that for any $x\in U^*\otimes V^*\otimes W^*$
we have a homomorphism
\begin{align*}
\psi_{j,x}\colon V\otimes \bw j U^* &\to W^*\otimes \bw {j+1} U^*,
\end{align*}
and $\psi_{j,x}$ depends linearly on~$x$. We call $\psi_{j,x}$ an \defi{exterior
flattening} of~$x$, as it generalizes the flattening of a tensor, as discussed
below.
\begin{defn}
Following~\cite[Defn.\ 1.1]{ev}, we define $\kappa_j(x)$ to be the rank of~$\psi_{j,x}$, and
we let $\kappa(x)$ denote the vector of $\kappa$-invariants $(\kappa_0(x),
\dots, \kappa_{m-1}(x))$.
\end{defn} 

More concretely, by choosing bases for the vector spaces, we can represent
$\psi_{j,x}$ as a matrix.  If $e_1, \dots, e_m$ is a basis for $U^*$, then a
basis for $\bw j U^*$ is given by the set of all $e_{i_1} \wedge \ldots
\wedge e_{i_j}$ for $1 \leq i_1 < \cdots < i_j \leq m$, and analogously for
$\bw {j+1}U^*$. For a fixed $u = \sum_{i=1}^m u_i e_i$ in~$U^*$, the map
$\bw j U^* \rightarrow \bw {j+1} U^*$ defined by $u' \mapsto u
\wedge u'$ will send $e_{i_1} \wedge \ldots \wedge e_{i_j}$ to $\sum_i u_i \,
e_i \wedge e_{i_1} \wedge \ldots \wedge e_{i_j}$. Thus, this map will be
represented in the above bases by a matrix whose entries are either $0$ or $\pm
u_i$.  The matrix for $\psi_{j,x}$ is the block matrix formed by replacing the
scalar $u_i$ with the matrix $A_i$ where $A_i$ are the matrices such that $x =
\sum_{i=1}^m e_i \otimes A_i$.

For example, if $m=4$, then $\psi_{j,x}$ are represented by the following
matrices (in suitable coordinates):
\begin{gather*}
\psi_{0,x}\colon V\otimes \bw 0 U^* \xrightarrow{\left(\begin{smallmatrix}
A_1\\A_2\\A_3 \\ A_4 \end{smallmatrix}\right)} 
W^*\otimes \bw 1 U^*,
\notag\\
\psi_{1,x}\colon V\otimes \bw 1 U^* \xrightarrow{
\left(\begin{smallmatrix}
0 & A_3 & -A_2 & 0 \\
-A_3 & 0 & A_1 & 0 \\
A_2 & -A_1 & 0 & 0 \\
A_4 & 0 & 0 & -A_1 \\
0 & A_4 & 0 & -A_2 \\
0 & 0 & A_4 & -A_3
\end{smallmatrix}\right)}
W^*\otimes \bw 2 U^*,
\\
\psi_{2,x}\colon V \otimes \bw 2 U^* \xrightarrow{\left(\begin{smallmatrix}
-A_4 & 0 & 0 & 0 & A_3 & -A_2 \\
0 & -A_4 & 0 & -A_3 & 0 & A_1 \\
0 & 0 & -A_4 & A_2 & -A_1 & 0 \\
A_1 & A_2 & A_3 & 0 & 0 & 0
\end{smallmatrix}\right)}
W^* \otimes \bw 3 U^*\\
\psi_{3,x}\colon V\otimes \bw 3 U^* \xrightarrow{\left(\begin{smallmatrix}
A_1&A_2&A_3&A_4 \end{smallmatrix}\right)}
W^*\otimes \bw 4 U^*.
\notag
\end{gather*}
The entries of these matrices are linear forms on $\PP(U^*\otimes V^*\otimes W^*)$, and 
the minors of these matrices $\psi_{j,x}$ are the ``explicit matrix equations''
alluded to in the introduction.

Note that for $j = 0$, the map $\psi_0 \colon V \otimes \bw 0 U^* \isom V \to
W^* \otimes U^*$ is
the homomorphism corresponding to~$x$ in the identification $U^* \otimes V^*
\otimes W^* \isom \Hom(V, U^* \otimes W^*)$. In the literature, the matrix for $\psi_0$ is referred to as a
``flattening'' of $x$ by grouping $W^*$ and~$U^*$.
Similarly,
$\psi_{m-1,x} \colon V \otimes \bw {m-1} U^* \to W^*$ is the flattening
formed by grouping $U^*$ and $V^*$, because $\bw {m-1} U^* \isom U$.

If $[x]\in \PP(U^{*}\otimes V^{*} \otimes W^{*})$ is a tensor, then the \defi{rank} of $[x]$ is the number $r$ in a minimal expression $x = u_{1}\otimes v_{1}\otimes w_{1} +\cdots+
u_{r}\otimes v_{r}\otimes w_{r}$, where $u_{i}\in U$, $v_{i}\in V$, and $w_{i}\in W$ for $1\leq i\leq r$.
The set of rank-one tensors is closed and equals the Segre variety $\Seg(\PP U^{*}\times \PP V^{*} \times \PP W^{*})$.
More generally, the Zariski closure of the set of tensors $[x]\in
\PP(U^{*}\otimes V^{*} \otimes W^{*})$ having rank at most~$r$ is the $r$th
secant variety of the Segre product, denoted $\sigma_{r}(\Seg(\PP
U^{*}\times \PP V^{*} \times \PP W^{*}))$.  
\begin{defn}\label{def:border-rank}
The \defi{border rank} of $[x]\in \PP(U^{*}\otimes V^{*} \otimes W^{*})$ is the minimal $r$ such that $[x]$ is in $\sigma_{r}(\Seg(\PP U^{*}\times \PP V^{*} \times \PP W^{*}))$.
\end{defn}

The following lemma generalizes the well-known fact that the ranks of the flattenings are bounded above by the tensor rank, and extends a result of Ottaviani~\cite[Thm.\ 3.2(i)]{Ottaviani07}. 
\begin{lemma}\label{lem:ranks}
If $x\in U^*\otimes V^*\otimes W^*$ has border rank at most~$r$, then $\kappa_j(x)\leq r \binom{m-1}{j}$.
\end{lemma}

\begin{proof}
Since $\kappa_j$ is defined in terms of a matrix rank,
an upper bound on $\kappa_j$ is a closed condition on the set of tensors.
It thus suffices to prove the statement with border rank replaced by rank.
As observed above, $\psi_{j,x}$ depends linearly on $x$, so it is
sufficient to assume that $x$ is an indecomposable tensor, and then show that
$\kappa_j(x) \leq \binom{m-1}{j}$. We can choose coordinates such that $x = e_1
\otimes A$, and $A$ is a matrix with only one non-zero entry. The non-zero rows
of the matrix for $\psi_{j,x}$ will correspond to those basis elements $e_{i_1}
\wedge \cdots \wedge e_{i_j} \in \bw j U^*$ such that $2 \leq i_1 < \cdots
< i_j \leq n$, each of which is sent to a multiple of $e_1 \wedge e_{i_1} \wedge
\cdots \wedge e_{i_j}$. Since there are $\binom{m-1}{j}$ such basis elements,
the rank of~$\psi_{j,x}$ is equal to $\binom{m-1}{j}$.
\end{proof}

The above lemma illustrates that the minors of the exterior flattenings provide 
equations which vanish on the secant variety of a Segre triple product.
We write $S^{\bullet}(U \otimes V \otimes W)$ to denote the polynomial ring on
the affine space $U^* \otimes V^* \otimes W^*$.

\begin{defn}\label{def:ideal-non-sym}
Let $c=(c_0, \dots, c_{m-1})$ be a vector of positive integers. We define
$I_{\kappa_i\leq c_i}$ to be the ideal generated by the $(c_i + 1) \times (c_i +
1)$-minors of $\psi_{j,x}$.
Similarly, we use the notation $I_{\kappa\leq c}$ for the
ideal generated by $I_{\kappa_i \leq c_i}$ for all $0 \leq i \leq m-1$.
Finally, we define $\Sigma_{\kappa_i\leq c_i}$ and $\Sigma_{\kappa\leq c}$ to be
the subschemes of $\PP (U^*\otimes V^*\otimes W^*)$ defined by the ideals $I_{\kappa_i\leq c_i}$ and~$I_{\kappa\leq
c}$ respectively.
\end{defn}

\begin{prop}\label{prop:vanish}
Fix $r \geq 1$.
If $c$ is the vector defined by $c_j=r \binom{m-1}{j}$ for $0 \leq j \leq m-1$,
then
\begin{equation*}
\sigma_r(\Seg(\PP U^{*} \times \PP V^* \times \PP W^*))\subseteq
\Sigma_{\kappa\leq c}.
\end{equation*}
\end{prop}

\begin{remark}\label{rmk:schur}
Fundamental in the construction of the exterior flattening $\psi_{j,x}$ was the inclusion
of~$U^*$ into $\bw {j} U \otimes \bw {j+1} U^*$. More generally, any
natural inclusion of $U^*$ into the tensor product of two representations would
yield an analogue of $\psi_{j,x}$ as well as analogues of Lemma~\ref{lem:ranks}
and Proposition~\ref{prop:vanish}.  For instance, from the inclusion
$U^*\subseteq S_{(2,1)}(U)\otimes S_{(2,1,1)}(U^*)$, we may associate to a
tensor~$x$ a homomorphism:
\begin{equation*}
V\otimes S_{(2,1)}(U^*)\to W^*\otimes S_{(2,1,1)}(U^*),
\end{equation*}
whose rank is at most $5$ times the border rank of~$x$. We restrict our attention
to the $\kappa$-invariants because these seem to provide particularly useful inequalities in our cases of interest.
However, an example of this generalized construction was introduced and applied
in~\cite[Thm. 1.1]{Ottaviani09} and has been further developed under the name of
Young flattening in~\cite{LanOtt}.
\end{remark}

\begin{example}
If $m=2$, then as stated above, $\kappa_0$ and~$\kappa_1$ are the ranks of the
flattenings formed by grouping $W^*$ with $U^*$ and $V^*$ with $U^*$,
respectively. The ideal of the $r$th secant variety is $I_{\kappa
\leq (r, r)}$~\cite[Theorem~1.1]{lw}.
\end{example}

\begin{example}
Let $m=3$ and suppose that $n=k$ is odd. Denote by~$X$ the Segre product $\Seg(\PP U^*\times \PP V^*\times
\PP W^*)$ in $\PP(U^*\otimes V^*\otimes W^*)$.  Then $I_{\kappa_1
\leq 3n-1}$ is a principal ideal generated by the determinant of $\psi_{1,x}$,
which defines the secant variety
$\sigma_{\frac{3n-1}{2}}(X)$~\cite[Lem.~4.4]{strassen}
(see also~\cite[Rmk.~3.3]{Ottaviani07}).
\end{example}

\begin{example}
The exterior flattening $\psi_{1,x}$ has also arisen in the study of totally symmetric tensors.  For instance, in the case $n=3$, the secant variety $\sigma_3(\nu_3(\mathbb P^2))\subseteq \mathbb P^9$ is a hypersurface defined by the Aronhold invariant.  Ottaviani has shown that this hypersurface is defined by any of the $8\times 8$ Pfaffians of the matrix representing $\psi_{1,x}$ specialized to symmetric tensors~\cite[Thm.\ 1.2]{Ottaviani09}.
\end{example}

However, the ideals $I_{\kappa\leq c}$ do \emph{not} equal the defining ideals
of secant varieties even in relatively simple cases.

\begin{example}
Let $n=m=k=3$ and let $Y$ be the image of $\mathbb P^2\times \mathbb
P^2\times \mathbb P^2\subseteq \mathbb P^{26}$ embedded by $\mathcal O(1,1,1)$.
By Proposition~\ref{prop:vanish}, we know that $I_{\kappa\leq (3,6,3)}$ vanishes
on~$\sigma_3(Y)$, but we claim that it is not the defining ideal.  Observe
that the conditions $\kappa_0\leq 3$ and $\kappa_2\leq 3$ are trivial, and hence
$I_{\kappa\leq (3,6,3)}=I_{\kappa_1\leq 6}$.  By definition, the ideal
$I_{\kappa_1\leq 6}$ is generated by the $7\times 7$ minors of~$\psi_{1,x}$.
However,~\cite[Thm.\ 1.3]{lw} produces degree~$4$ equations which vanish on
$\sigma_3(Y)$, and since
$I_{\kappa_1\leq 6}$ is generated in degree~$7$, we see
that it does not equal the defining ideal of $\sigma_3(Y)$. 
\end{example}

We now study our matrix equations from the perspective of representation theory,
which connects them to previous work on secant varieties of Segre-Veronese
varieties. The representation theory of our ideals $I_{\kappa_i \leq c_i}$ will
also be necessary in the proof of Lemma~\ref{lem:subspace-projection}.

Since the ideals $I_{\kappa_i\leq c_i}$ are invariant under the natural action
of $\GL(U) \times \GL(V) \times \GL(W)$, their generators can be decomposed
as direct sums of irreducible representations of that group.  Each polynomial
representation of $\GL(U) \times \GL(V) \times \GL(W)$ is of the form $S_{\mu}U
\otimes S_{\nu}V \otimes S_{\omega}W$ where $S_{\mu}U$, $S_{\nu}V$, and $S_{\omega}W$ are the Schur modules indexed by partitions $\mu$, $\nu$, and~$\omega$
with at most $m$, $n$, and~$k$ parts, respectively, where if $\pi =
(\pi_{1},\dots,\pi_{s})$ is a partition with $\pi_{1}\geq\pi_{2}\geq \dots \geq
\pi_{s} > 0$, then we say that $s$ is the number of parts of~$\pi$.  For the
summands of the degree~$d$ part of~$I_{\kappa_i \leq c_i} \subset S^d(U\otimes V
\otimes W)$, the partitions will always be partitions of~$d$.  For general
background on Schur modules see~\cite{FultonHarris}.

\begin{lemma}\label{lem:kapp2schur}
For each $j$ and $c_{j}$, there is a $\GL(U)\times \GL(V)\times \GL(W)$-equivariant map
\begin{equation*}
\Phi_j\colon \ \bw {c_j+1}\left( V \otimes \bw {j}U^{*}\right) \otimes \bw
{c_j+1}\left(W\otimes \bw {j+1} U\right) \to S^{c_j+1}\left(U\otimes V\otimes
W\right),
\end{equation*}
whose image equals the vector space of generators of $I_{\kappa_j\leq c_j}$.  In
particular, every irreducible representation arising in the Schur module
decomposition of the generators of $I_{\kappa_j\leq c_j}$ must be a submodule of
both the source and target of~$\Phi_j$.
\end{lemma}
\begin{proof}
Consider the map
\[
\psi_{j,x}\colon V\otimes \bw j U^* \to W^*\otimes \bw {j+1} U^*.
\]
After choosing bases of $U$, $V$, and~$W$, we may think of $\psi_{j,x}$ as a matrix of linear forms in $S^\bullet(U\otimes V\otimes W)$.
Taking the $(c_j+1)\times (c_j+1)$-minors of $\psi_{j,x}$ then determines the
map~$\Phi_j$.  More concretely, our choice of bases for $U$, $V$, and~$W$ determines a natural basis for the source of $\Phi_j$ consisting of indecomposable tensors; we define the map $\Phi_j$ by sending a basis element to the corresponding minor of the matrix $\psi_{j,x}$.  Since the ideal $I_{\kappa_j\leq c_j}$ is defined as the ideal generated by the
image of $\Phi_j$,  the lemma follows from Schur's Lemma.
\end{proof}
 When $j=0$, it is straightforward to compute the Schur module decomposition of $I_{\kappa_0\leq c_0}$, as illustrated by the following example.
\begin{example}\label{ex:333}
The map
\begin{equation*}
\psi_{0,x}\colon V\otimes \bw 0 U^* \rightarrow 
W^*\otimes \bw 1 U^*
\end{equation*}
is a flattening of the tensor $x$ by grouping $U^*$ and $W^*$.  
As representations, the minors of $\psi_{0,x}$ decompose into irreducibles using the skew Cauchy formula \cite[p.~80]{FultonHarris}
\begin{equation*}
\bw{c_{0}+1} V  \otimes \bw{c_{0}+1}( W\otimes U) = \bw{c_{0}+1} V \otimes \left(\displaystyle{\bigoplus_{|\lambda|=c_{0}+1}} S_{\lambda}W \otimes S_{\lambda'}U \right)
,\end{equation*}
where $\lambda$ ranges over all partitions of $c_{0}+1$
and $\lambda'$ is the conjugate partition.

For instance, let $n=m=k=3$ and consider the generators of $I_{\kappa_0\leq 2}$.
This is a vector space of cubic polynomials, and by Lemma~\ref{lem:kapp2schur},
it must be the module
\[
\bw 3 V\otimes \big((S_{3}W\otimes S_{1,1,1}U) \oplus (S_{2,1}W\otimes
S_{2,1}U) \oplus (S_{1,1,1}\otimes S_3U)\big).
\]
After distributing, each irreducible module is the tensor product of three Schur
functors applied to $U$, $V$, and~$W$ respectively, and we can thus drop the vector
spaces and the tensor products from our notation, replacing $S_{\mu}U\otimes S_{\nu}V\otimes S_{\omega}W$ with $S_{\mu}S_{\nu}S_{\omega}$ without any ambiguity. Thus, we
rewrite this module as:
\[
S_{1,1,1}S_{1,1,1}S_3\oplus S_{2,1}S_{1,1,1}S_{2,1}\oplus S_{3}S_{1,1,1}S_{1,1,1}.
\]
The dimension of this space is $10+64+10=84$, which equals the number of maximal minors of the
$3\times 9$ matrix $\psi_{0,x}$.
\end{example}

When $j > 0$, the existence of the dual vector space $U^*$ in the source of $\psi_{j,x}$ makes
finding the Schur module decomposition of $I_{\kappa_i \leq c_i}$ more subtle.
In Proposition~\ref{prop:kappa1nonsym:decomp} below, we provide
an upper bound for the Schur module decomposition of $I_{\kappa_1\leq c_1}$ in the case that
$\dim U=3$.  To state the formula precisely, we first recall some notation.

For any vector space~$A$, the Littlewood-Richardson formula is
\begin{equation}\label{eq:LR:product}
S_{\lambda}A \otimes S_{\mu} A = \bigoplus_{|\pi| = |\lambda|+|\mu|}
S_{\pi}A ^{\oplus c^{\pi}_{\lambda,\mu}}
,\end{equation}
where the multiplicities $c^{\pi}_{\lambda,\mu}$ are the
\defi{Littlewood-Richardson numbers}.
For two vectors spaces $A$ and $B$, we use the outer plethysm formula
\begin{equation}\label{eq:out:plethysm}
S_{\pi}(A\otimes B) = \bigoplus_{|\lambda|+|\mu|=|\pi| }
(S_{\lambda}A \otimes S_{\mu}B) ^{\oplus K_{\pi,\lambda,\mu}}
\end{equation}
to define the \defi{Kronecker coefficients} $K_{\pi,\lambda,\mu}$.

\begin{remark}\label{rmk:repdual}
In Propositions~\ref{prop:kappa1nonsym:decomp},
\ref{prop:kappa0:decomp} and~\ref{prop:kappa1sym:decomp}, we will use the fact
that as $\GL(U)$ modules, 
$S_{\pi}U^{*}\otimes (\bw {m}U)^{l} \cong S_{l^{m}-\pi} U $,
where $l^m$ denotes the partition $(l, \ldots, l)$.
We caution that the entries in $l^{m} -\pi = (l - \pi_m, \ldots, l-\pi_1)$ are
reversed.
\end{remark}

\begin{prop}\label{prop:kappa1nonsym:decomp} Let $\dim(U) =3$.
For any Schur module $S_{\pi}U \otimes S_{\lambda}V \otimes S_{\mu}W$, let
$\lambda'$ and $\mu'$ denote the conjugate partitions of $\lambda$ and~$\mu$
respectively, and let $(3)^{c_0 + 1} - \pi$ be the difference as in
Remark~\ref{rmk:repdual}.
If $S_\pi U \otimes S_\lambda V \otimes S_\mu W$ occurs in the decomposition of
$\left(I_{\kappa_{1}\leq
c_{1}}\right)_{c_{1}+1}$ from Definition~\ref{def:ideal-non-sym}, then $\pi$,
$\lambda'$, and~$\mu'$ have at most $3$~parts, and the multiplicity of $S_\pi U
\otimes S_\lambda V \otimes S_\mu W$ is at most the minimum of
$c^{(3)^{c_{0}+1}-\pi}_{\lambda',\mu'}$
and~$K_{\pi,\lambda,\mu}$.
\end{prop}

Computations with the software package
LiE~\cite{lie}
suggest that the decomposition of $\left(I_{\kappa_{1}\leq
c_{1}}\right)_{c_{1}+1}$ may equal the upper bound of
Proposition~\ref{prop:kappa1nonsym:decomp}, as is the case in
Example~\ref{ex:non-sym-333}.

\begin{proof}[Proof of Proposition~\ref{prop:kappa1nonsym:decomp}]
Using Lemma~\ref{lem:kapp2schur}, the $(c_{1}+1) \times (c_{1}+1)$-minors of
$\psi_{1,x}$ belong to the common submodules of the polynomials
$S^{c_{1}+1}(U\otimes V\otimes W)$ and the domain of $\Phi_j$, which we can
rewrite using the Cauchy skew formula:
\begin{equation}\label{eq1:prop:kappa1:decomp}
\left(\bigoplus_{|\lambda|=c_{1}+1} S_{\lambda}V \otimes
S_{\lambda'}U^{*} \right)
\otimes \left(\bigoplus_{|\mu|=c_{1}+1} S_{\mu}W \otimes S_{\mu'}\big(\bw{2}U\big) \right)
.\end{equation}
Here we note that $\lambda'$ and $\mu'$ must have no more than $3$ parts or else the summand is zero.

We focus on the $U$ factor and compute
\begin{align*}
S_{\lambda'}U^{*}\otimes S_{\mu'}(\bw {2}U)
&\cong S_{\lambda'}U^{*}\otimes S_{\mu'}(U^{*}) \otimes (\bw 3 U)^{c_0 + 1}
&&\mbox{because $\dim U = 3$}\\
&\cong 
\bigoplus_{\lvert\nu\rvert = 2(c_{0}+1)}(S_{\nu}U^{*})^{\oplus c^{\nu}_{\lambda',\mu'}}
\otimes (\bw 3 U)^{c_0 + 1}
&&\mbox{by \eqref{eq:LR:product}} \\
&\cong
\bigoplus_{\lvert\nu\rvert = 2(c_{0}+1)}(S_{(c_{0}+1)^{3}-\nu}\,U)^{\oplus c^{\nu}_{\lambda',\mu'}} 
&&\mbox{by Remark~\ref{rmk:repdual} }\\
&\cong\bigoplus_{\lvert\pi\rvert = c_{0}+1}(S_{\pi}U)^{\oplus c^{(c_{0}+1)^{3}-\pi}_{\lambda',\mu'}} 
&&\mbox{by taking $\pi = (c_0 + 1)^3 - \nu$}
.\end{align*}
Therefore expression \eqref{eq1:prop:kappa1:decomp} becomes 
\[
\bigoplus_{|\lambda|=|\mu|=|\pi| =c_{0}+1} S_{\pi}U \otimes S_{\lambda}V \otimes S_{\mu}W ^{\oplus c^{(c_{0}+1)^{3}-\pi}_{\lambda',\mu'}}
.\]

Now we must decide which irreducible modules
occur as a submodule of $S^{c_1+1}(U \otimes V \otimes W)$. For this, we
decompose the space of polynomials using the Cauchy formula and the outer
plethysm formula~\eqref{eq:out:plethysm}:
\[S^{c_{1}+1}(U\otimes V\otimes W)
\cong \bigoplus_{|\pi|=c_{1}+1} S_{\pi} U \otimes S_{\pi}(V \otimes W)
\cong \bigoplus_{|\pi|=|\lambda|=|\mu|=c_{0}+1} (S_{\pi} U \otimes S_{\lambda}V
\otimes S_{\mu}W) ^{\oplus K_{\pi,\lambda,\mu}}
.\]
The proposition statement follows by Lemma~\ref{lem:kapp2schur}.
\end{proof}

\begin{example}\label{ex:non-sym-333}
\begin{figure}
\begin{align*}
&
%     this is case degree =
%     2
%     [1,1,0,1,1,0,2,0,0]
%     [1,1,0,2,0,0,1,1,0]
%     [2,0,0,1,1,0,1,1,0]
%     [2,0,0,2,0,0,2,0,0]
\begin{array}{ll}\left(I_{\kappa_{1}\leq 1}\right)_2 =&
(\mathfrak{S}_{3}\!\cdot\! S_{1,1}S_{1,1}S_{2}) \oplus S_{2}S_{2}S_{2} \end{array}
\\
&
%     this is case degree =
%     3
%     [0,0,0,2,1,0,2,1,0]
%     [2,1,0,0,0,0,2,1,0]
%     [2,1,0,2,1,0,0,0,0]
%     [2,1,0,2,1,0,2,1,0]
%     [2,1,0,2,1,0,3,0,0]
%     [2,1,0,3,0,0,2,1,0]
%     [3,0,0,0,0,0,0,0,0]
%     [3,0,0,2,1,0,2,1,0]
\begin{array}{ll}\left(I_{\kappa_{1}\leq 2}\right)_3 =& 
(\mathfrak{S}_{3}\!\cdot\! S_{1,1,1}S_{2,1}S_{2,1}) \oplus S_{2,1}S_{2,1}S_{2,1} \oplus
(\mathfrak{S_{3}}\!\cdot\! S_{2,1}S_{2,1}S_{3}) \oplus S_{3}S_{1,1,1}S_{1,1,1} 
\end{array}
\\
&
%     this is case degree =
%     4
%     [2,2,0,2,2,0,2,2,0]
%     [2,2,0,1,0,0,1,0,0]
%     [2,2,0,1,0,0,3,1,0]
%     [2,2,0,3,1,0,1,0,0]
%     [2,2,0,3,1,0,3,1,0]
%     [1,0,0,2,2,0,1,0,0]
%     [1,0,0,2,2,0,3,1,0]
%     [1,0,0,1,0,0,2,2,0]
%     [1,0,0,1,0,0,1,0,0]
%     [1,0,0,3,1,0,2,2,0]
%     [1,0,0,3,1,0,3,1,0]
%     [3,1,0,2,2,0,1,0,0]
%     [3,1,0,2,2,0,3,1,0]
%     [3,1,0,1,0,0,2,2,0]
%     [3,1,0,1,0,0,1,0,0]
%     [3,1,0,1,0,0,3,1,0]
%     [3,1,0,3,1,0,2,2,0]
%     [3,1,0,3,1,0,1,0,0]
%     [4,0,0,2,2,0,2,2,0]
%     [4,0,0,1,0,0,1,0,0]
\begin{array}{ll}\left(I_{\kappa_{1}\leq 3}\right)_4 =& 
S_{2,2}S_{2,2}S_{2,2}\oplus
(\mathfrak{S}_{3}\!\cdot\! S_{2,2}S_{2,1,1}S_{2,1,1}) \oplus 
(\mathfrak{S}_{3}\!\cdot\! S_{2,2}S_{2,1,1}S_{3,1}) \oplus 
(\mathfrak{S}_{3}\!\cdot\! S_{2,2}S_{3,1}S_{3,1}) \\ &
{}\oplus 
S_{2,1,1}S_{2,1,1}S_{2,1,1} \oplus 
(\mathfrak{S}_{3}\!\cdot\! S_{2,1,1}S_{3,1}S_{3,1}) \oplus 
S_{3,1}S_{2,1,1}S_{2,1,1} \oplus
S_{4}S_{2,2}S_{2,2} \\ &
{}\oplus 
S_{4}S_{2,1,1}S_{2,1,1}
\end{array}
\\
&
%     this is case degree =
%     5
%     [1,1,0,1,1,0,1,1,0]
%     [1,1,0,1,1,0,2,0,0]
%     [1,1,0,3,2,0,3,2,0]
%     [1,1,0,3,2,0,2,0,0]
%     [1,1,0,2,0,0,1,1,0]
%     [1,1,0,2,0,0,3,2,0]
%     [3,2,0,1,1,0,1,1,0]
%     [3,2,0,1,1,0,3,2,0]
%     [3,2,0,1,1,0,2,0,0]
%     [3,2,0,3,2,0,1,1,0]
%     [3,2,0,3,2,0,2,0,0]
%     [3,2,0,2,0,0,1,1,0]
%     [3,2,0,2,0,0,3,2,0]
%     [3,2,0,2,0,0,2,0,0]
%     [2,0,0,1,1,0,1,1,0]
%     [2,0,0,1,1,0,3,2,0]
%     [2,0,0,1,1,0,2,0,0]
%     [2,0,0,3,2,0,1,1,0]
%     [2,0,0,3,2,0,3,2,0]
%     [2,0,0,2,0,0,1,1,0]
%     [2,0,0,2,0,0,2,0,0]
%     [4,1,0,1,1,0,1,1,0]
%     [4,1,0,1,1,0,3,2,0]
%     [4,1,0,1,1,0,2,0,0]
%     [4,1,0,3,2,0,1,1,0]
%     [4,1,0,2,0,0,1,1,0]
%     [5,0,0,1,1,0,1,1,0]
\begin{array}{ll}\left(I_{\kappa_{1}\leq 4}\right)_5 =& 
S_{2,2,1}S_{2,2,1}S_{2,2,1}\oplus
(\mathfrak{S}_{3}\!\cdot\! S_{2,2,1}S_{2,2,1}S_{3,1,1})\oplus
(\mathfrak{S}_{3}\!\cdot\! S_{2,2,1}S_{3,2}S_{3,2})\oplus
(\mathfrak{S}_{3}\!\cdot\! S_{2,2,1}S_{3,2}S_{3,1,1}) \\ &
{}\oplus (\mathfrak{S}_{3}\!\cdot\! S_{3,2}S_{3,2}S_{3,1,1})\oplus
S_{3,2}S_{2,2,1}S_{2,2,1}\oplus
S_{3,2}S_{3,1,1}S_{3,1,1}\oplus
S_{3,1,1}(\mathfrak{S}_{2}\!\cdot\! S_{2,2,1}S_{3,1,1})
\\ &
{}\oplus     
S_{3,1,1}S_{3,1,1}S_{3,1,1}\oplus     
S_{4,1}S_{2,2,1}S_{2,2,1}\oplus     
S_{4,1}(\mathfrak{S}_{2}\!\cdot\! S_{2,2,1}S_{3,2})
\oplus     
S_{4,1}(\mathfrak{S}_{2}\!\cdot\! S_{2,2,1}S_{3,1,1})
\\&
{}\oplus     
S_{5}S_{2,2,1}S_{2,2,1}
\end{array}
\\
&
%     this is case degree =
%     6
%     [0,0,0,3,3,0,3,3,0]
%     [0,0,0,2,1,0,2,1,0]
%     [3,3,0,2,1,0,2,1,0]
%     [2,1,0,0,0,0,2,1,0]
%     [2,1,0,3,3,0,2,1,0]
%     [2,1,0,2,1,0,0,0,0]
%     [2,1,0,2,1,0,3,3,0]
%     [2,1,0,2,1,0,2,1,0]
%     [4,2,0,0,0,0,0,0,0]
%     [4,2,0,0,0,0,2,1,0]
%     [4,2,0,2,1,0,0,0,0]
%     [4,2,0,2,1,0,2,1,0]
%     [3,0,0,0,0,0,3,3,0]
%     [3,0,0,0,0,0,2,1,0]
%     [3,0,0,3,3,0,0,0,0]
%     [3,0,0,2,1,0,0,0,0]
%     [3,0,0,2,1,0,2,1,0]
%     [5,1,0,0,0,0,2,1,0]
%     [5,1,0,2,1,0,0,0,0]
%     [6,0,0,0,0,0,0,0,0]
\begin{array}{ll}
\left(I_{\kappa_{1}\leq 5}\right)_6 =& S_{2,2,2}S_{3,3}S_{3,3} \oplus
(\mathfrak{S}_{3}\!\cdot\! S_{2,2,2}S_{3,2,1}S_{3,2,1}) \oplus
(\mathfrak{S}_{3}\!\cdot\! S_{3,3}S_{3,2,1}S_{3,2,1}) \oplus          
(S_{3,2,1}S_{3,2,1}S_{3,2,1})^{\oplus 2} \\ &
{}\oplus
S_{4,2}S_{2,2,2}S_{2,2,2} \oplus
S_{4,2}(\mathfrak{S}_{2}\!\cdot\! S_{2,2,2}S_{3,2,1} ) \oplus
S_{4,2}S_{3,2,1}S_{3,2,1} \oplus
S_{4,1,1}(\mathfrak{S}_{2}\!\cdot\! S_{2,2,2}S_{3,3} ) \\ &
{}\oplus
S_{4,1,1}(\mathfrak{S}_{2}\!\cdot\! S_{2,2,2}S_{3,2,1} ) \oplus     
S_{4,1,1}S_{3,2,1}S_{3,2,1} \oplus     
S_{5,1}(\mathfrak{S}_{2}\!\cdot\! S_{2,2,2}S_{3,2,1} ) \oplus          
S_{6}S_{2,2,2}S_{2,2,2}
\end{array}
\\
%     this is case degree =
%     7
%     [2,2,0,2,2,0,2,2,0]
%     [2,2,0,1,0,0,1,0,0]
%     [1,0,0,2,2,0,2,2,0]
%     [1,0,0,2,2,0,1,0,0]
%     [1,0,0,1,0,0,2,2,0]
%     [3,1,0,2,2,0,1,0,0]
%     [3,1,0,1,0,0,2,2,0]
%     [3,1,0,1,0,0,1,0,0]
%     [4,0,0,1,0,0,1,0,0]
&
\begin{array}{ll}\left(I_{\kappa_{1}\leq 6}\right)_7 =& 
S_{3,3,1}S_{3,3,1}S_{3,3,1} \oplus 
(\mathfrak{S}_{3}\!\cdot\! S_{3,3,1}S_{3,2,2}S_{3,2,2} )\oplus 
S_{3,2,2}S_{3,3,1}S_{3,3,1} \oplus 
S_{4,2,1}(\mathfrak{S}_{2}\!\cdot\! S_{3,3,1}S_{3,2,2}) \\ &
{}\oplus 
S_{4,2,1}S_{3,2,2}S_{3,2,2} 
\oplus  S_{5,1,1}S_{3,2,2}S_{3,2,2}
\end{array}
\\
%     this is case degree =
%     8
%     [1,1,0,1,1,0,1,1,0]
%     [2,0,0,1,1,0,1,1,0]
&
\begin{array}{ll}
\left(I_{\kappa_{1}\leq 7}\right)_8 =& S_{3,3,2}S_{3,3,2}S_{3,3,2}\oplus S_{4,2,2}S_{3,3,2}S_{3,3,2}
\end{array}
\\
&
\begin{array}{ll}\left(I_{\kappa_{1}\leq 8}\right)_9 =& S_{3,3,3}S_{3,3,3}S_{3,3,3}\end{array}
\end{align*}
\caption{The Schur module decompositions of $\left(I_{\kappa_{1}\leq c_{1}}\right)_{c_{1}+1}$ in the $\mathcal{O}(1,1,1)$ case when $k=m=n=3$.}\label{fig:kappa1notsym}
\end{figure}
Let $n=m=k=3$.  Using LiE~\cite{lie}, we computed
every decomposition of $(I_{\kappa_{1}\leq c_{1}})_{c_{1}+1}$ using
Proposition \ref{prop:kappa1nonsym:decomp}. These decompositions appear in
Figure~\ref{fig:kappa1notsym}. To save space we omit the notation of vector
spaces and tensor products, as in Example~\ref{ex:333}.  Further, we use the
notation $\mathfrak{S}_{s}$ to indicate the direct sum over the (non-redundant)
permutations of the subsequent $s$~Schur modules.
The dimensions of the modules found in Figure~\ref{fig:kappa1notsym} are
\[\begin{tabular}{c|c|c|c|c|c|c|c|cc}
$c_{1}$  & 1 & 2 & 3 & 4 & 5 & 6 & 7 & 8\\ \hline
$\dim\;(I_{\kappa_{1}\leq c_{1}})_{c_{1}+1}$ & 378 &2634 &8910& 12420& 7011& 1296& 81 & 1
\end{tabular}
.\]
Since these dimensions match the dimensions of the space of minors of
$\psi_{1,x}$, as computed in Macaulay2~\cite{M2}, all of the modules must be in the space
of minors.

For $c_1 = 6$, we have $\dim\;(I_{\kappa_{1}\leq
6})_{7}=1296=\binom{9}{7}^2$, and hence we see that all $7\times 7$-minors of $\psi_{1,x}$
are linearly independent. By contrast, if $c_1=5$, the fact that $\dim_{\CC}
(I_{\kappa_{1}\leq 5})_{6}=7011<7056=\binom{9}{6}^2$ tells us that the $6\times
6$ minors of $\psi_{1,x}$ are not all linearly independent. For example, the
upper right and lower left $6\times 6$ minors of
\begin{equation*}
\psi_{1,x}=\begin{pmatrix}
0 & A_3 & -A_2 \\
-A_3 & 0 & A_1 \\
A_2 & -A_1 & 0
\end{pmatrix}
\end{equation*}
are both equal to $\det(A_1) \cdot \det(A_3)$.\end{example}

In the next section we impose partial symmetry on our $3$-tensors.
We remark that we could impose other types of symmetry and this would lead
to different investigations.  For instance, we could restrict attention to
$3$-tensors in any of the following cases: $S^3(U^*)$, $U^*\otimes \bw 2 V^*$, $\bw 3 V^*$, or $S_{2,1}U^*$.  In these cases, it would be straightforward to prove analogues of Lemma~\ref{lem:ranks} and Proposition~\ref{prop:vanish}.  However, if we hope to produce the ideal defining the appropriate secant varieties, then it is less obvious how to generalize Definition~\ref{def:ideal-non-sym}.  It might be interesting to investigate the secant varieties of these other special types of $3$-tensors.

%%%%%%%%%%%%%%%%%%%%%%%%
%%%%%%%%%%%%%%%%%%%%%%%%
\section{The \texorpdfstring{$\kappa$}{k}-invariant for partially symmetric $3$-tensors}\label{sec:kappa-semi}
%%%%%%%%%%%%%%%%%%%%%%%%
%%%%%%%%%%%%%%%%%%%%%%%%

For the rest of the paper, we take $W = V$ and focus on partially symmetric
$3$-tensors $x\in U^*\otimes S^2V^* \subset U^* \otimes V^* \otimes V^*$.
By this latter inclusion, we may extend the definition of $\kappa_j(x)$ to
partially symmetric tensors. 
Not only does the $\kappa$-invariant provide a bound for
the rank of $x$, but also for the \defi{partially symmetric rank}, which is defined
as the minimal $r$ such that $x = \sum_{i=1}^r u_i \otimes v_i \otimes v_i$, for
some $u_i \in U^*$ and $v_i \in V^*$.
The set of rank-one partially symmetric tensors is known as the
Segre-Veronese variety of $\PP U^* \times \PP V^*$ embedded by $\Osh(1,2)$.
Therefore, the Zariski closure of the set of partially symmetric tensors of rank at
most~$r$ is the $r$th secant variety of the Segre-Veronese variety. We have the
following analogue of Definition~\ref{def:border-rank}.

\begin{defn}\label{def:partially sym-border-rank}
The \defi{partially symmetric border rank} of $[x]\in \PP(U^{*}\otimes S^{2}V^{*})$ is the minimal $r$ such that $[x]$ is in $\sigma_{r}(\Seg(\PP U^{*}\times v_{2}(\PP V^{*})))$.
\end{defn}

Providing an analogue of the equations from Definition~\ref{def:ideal-non-sym} is a bit more subtle in the case of partially symmetric $3$-tensors.
In fact, it is necessary to
refine the equations if we hope to produce ideals which are radical.
To see this, consider the case where $x$ is a
partially symmetric $3\times n \times n$ tensor. For such an $x$, the
matrix representing $\psi_{1,x}$ has the form
\[
\psi_{1,x}=
\begin{pmatrix}
0&A_3&-A_2\\ -A_3 & 0 & A_1 \\
A_2&-A_1&0
\end{pmatrix}
\]
where the $A_i$ are symmetric $n\times n$-matrices.  Since $\psi_{1,x}$ is a
skew-symmetric matrix, all of the principal minors
in $I_{\kappa_1 \leq c_1}$ are squares.

More generally, if $m = 4j+3$, then $\psi_{2j+1, x} \colon V \otimes
\bw {2j+1} U^* \rightarrow V^* \otimes \bw {2j+2} U^*$ is
represented by a skew-symmetric matrix in appropriate coordinates.  Thus, the
condition that $\psi_{2j+1,x}$ has rank at most an even integer $c_{2j+1}$ is
defined algebraically by the principal $(c_{2j+1}+2) \times (c_{2j+1}+
2)$-Pfaffians of $\psi_{2j+1,x}$. These Pfaffians have degree $c_{2j+1}/2 +1$,
whereas the $(c_{2j+1}+1) \times (c_{2j+1}+ 1)$-minors have degree $c_{2j+1} +
1$.

To encode this skew-symmetry into our matrix equations in the case of partially symmetric tensors, we introduce the
following analogue of Definition~\ref{def:ideal-non-sym}.

\begin{defn}\label{def:ideal-sym}
Let $I_{\kappa_j \leq c_j}$ in~$S^\bullet(U \otimes S^2 V)$ denote the
ideal generated by the $(c_j +2) \times (c_j +2)$-Pfaffians of $\psi_{j,x}$, if
$j = (m-1)/2$, $j$ is an odd integer, and $c_j$ is even. Otherwise,
$I_{\kappa_j\leq c_j}$ denotes the specialization of the
ideal in Definition~\ref{def:ideal-non-sym}.
As in
Definition~\ref{def:ideal-non-sym}, for a vector~$c$, $I_{\kappa \leq c}$ is
defined to be the
ideal generated by the $I_{\kappa_j \leq c_j}$ for all~$j$, and
$\Sigma_{\kappa_i\leq c_i}$ and $\Sigma_{\kappa\leq c}$ are
the subschemes of $\PP (U^*\otimes S^2V^*)$ defined by $I_{\kappa_i\leq
c_i}$ and~$I_{\kappa\leq c}$ respectively.
\end{defn}

Note that for partially symmetric tensors~$x$, we have $\kappa_j(x) =
\kappa_{m-1-j}(x)$ and likewise $I_{\kappa_j(x) \leq c_j} = I_{\kappa_{m-1-j}
\leq c_j}$ for
all~$j$.
With notation in Definition~\ref{def:ideal-sym}, we also obtain the following analogue of Proposition~\ref{prop:vanish}.
\begin{prop}\label{prop:vanish:semi}
Fix $r \geq 1$.  Let $X$ be the Segre-Veronese variety of $\PP (U^*) \times \PP(V^*)$ in $\PP(U^*\otimes S^2V^*)$ and let
 $c$ be the vector defined by $c_j=r\binom{m-1}{j}$ for $0 \leq j \leq m-1$.  Then
$\sigma_r(X)\subseteq
\Sigma_{\kappa\leq c}$.
\end{prop}

\begin{remark}
Although the rest of the paper concerns partially symmetric $3 \times n \times n$ tensors, we
note that the equations given in Definition~\ref{def:ideal-sym} would be
insufficient to generate the ideal of the secant varieties for $m \geq 4$.
For example, consider the case of partially symmetric $4\times n\times n$ tensors.
If $x$ is such a tensor, then the matrix representing $\psi_{2,x}$
has the form
\begin{equation*}\psi_{2,x}=
\begin{pmatrix}
-A_4&0&0&\mathbf{0}&\mathbf{A_3}&\mathbf{-A_2}\\
0&-A_4&0&\mathbf{-A_3}&\mathbf{0}&\mathbf{A_1}\\
0&0&-A_4&\mathbf{A_2}&\mathbf{-A_1}&\mathbf{0}\\
A_1&A_2&A_3&0&0&0
\end{pmatrix}
\end{equation*}
where each $A_i$ is an $n\times n$ symmetric matrix. If $x$ has border rank at
most $r$, then $\psi_{1,x}$ will have rank at most $3r$ by
Proposition~\ref{prop:vanish:semi}. However, the bold submatrix in the upper
right will have rank at most $2r$. Moreover, since the bold submatrix is
skew-symmetric, the condition that it has rank at most~$2r$ is given by the
vanishing of its $(2r+2)\times (2r+2)$-principal Pfaffians. Thus, the defining
ideal of the $r$th secant variety must contain these Pfaffians, as well as $3$
other sets of Pfaffians which arise by symmetry. Since the Pfaffians have
degree $r+1$, they can not be in the ideal of the $(3r+1) \times (3r+1)$-minors.

In effect, these Pfaffians amount to the generators of $I_{\kappa_1 \leq 2r}$
applied to a $3 \times n \times n$ subtensor. In the literature on tensors, this
process for producing equations on larger tensors by applying known equations to all
subtensors is known as \defi{inheritance}. See~\cite[\S2.1]{lm2} for a precise
definition in the language of representation theory. The above analysis shows
how the inheritance of $\kappa$-equations can produce new equations beyond the
$\kappa$-equations themselves. 
\end{remark}

\begin{prop}\label{prop:kappa0:decomp}
As a Schur module, we have the following decomposition of the generators $I_{\kappa_0\leq r}$ into irreducible representations of $\GL(U)\times \GL(V)$
\[
(I_{\kappa_{0}\leq r})_{r+1} = \bigoplus_{|\pi|=r+1} S_{\pi}U \otimes S_{\pi' + 1^{r+1}}V,
\]
where $\pi'$ is the conjugate partition to $\pi$, and $1^{r+1} = (1, \ldots, 1)$
is the partition with $r+1$ parts.
\end{prop}

In the proof, we will need the following observation.

\begin{lemma}\label{lem:mostparts}
Suppose $\pi$ is a partition of $d$ and suppose $A$ is a vector space.  If $S_{\lambda}A$ is a module occurring in the decomposition of $S_{\pi}(S^{2} A)$ then $\lambda$ has at most $d$ parts.
\end{lemma}

\begin{proof}
Since $\pi$ is a partition of~$d$, we have an inclusion
$S_{\pi}(S^{2}A) \subset (S^{2}A)^{\otimes d}$.
By inductively applying the Pieri formula to $(S^{2}A)^{\otimes d} =
(S^{2}A)^{\otimes d-1}\otimes S^{2}A$, we see that every module in the
decomposition of $(S^2A)^{\otimes d}$ can have at most $d$~parts.
\end{proof}  

\begin{proof}[Proof of Proposition \ref{prop:kappa0:decomp}]
After choosing bases of $U$ and $V$, we may view the map $\psi_{0,x}$
as a matrix of linear forms in $S^\bullet(U\otimes S^2V)$.
By Lemma~\ref{lem:kapp2schur}, $(I_{\kappa_{0}\leq r})_{r+1}$ is
the image of the $\GL(U)\times \GL(V)$-equivariant morphism
\begin{equation*}
\Phi_0\colon \bw {r+1} V \otimes \bw {r+1} \left(U\otimes V \right) \to
S^{r+1}(U\otimes S^{2}(V)),
\end{equation*}
which sends an indecomposable basis element in the source to the corresponding minor
 in the polynomial ring.  The first step of our proof is to show
that only those representations of the form $S_{\pi}U \otimes S_{\pi' +
1^{r+1}}V$ appear in both the source and target of~$\Phi_0$.  By Schur's Lemma,
this will provide a necessary condition on the representations which can appear
in the decomposition of~$(I_{\kappa_{0}\leq r})_{r+1}$.  The second step of our
proof is to show that each such representation actually arises; for this, we
produce an explicit nonzero minor of $\psi_{0,x}$ that is in the image
of~$\Phi_{0}$ restricted to  $S_{\pi}U \otimes S_{\pi' + 1^{r+1}}V$, so that
$\Phi_{0}$ restricted to  $S_{\pi}U \otimes S_{\pi' + 1^{r+1}}V$
is nonzero.

For the first step, 
suppose that $S_{\pi}U \otimes S_{\lambda}V$ is a module in
$S^{r+1}(U\otimes S^{2}V)$.  
If we apply the Cauchy decomposition formula to $S^d(U \otimes
S^2V )$, and consider the resulting modules as $\GL(U)$-representations,
then we must have
$S_{\pi}U \otimes S_{\lambda}V$ contained in the summand $S_{\pi}U \otimes
S_{\pi}(S^{2}V)$. In particular we must have $S_{\lambda}V \subset
S_{\pi}(S^{2}V)$. 
Therefore, by Lemma \ref{lem:mostparts},
$\lambda$ has at most $r+1$ parts.

On the other hand, we can use the skew Cauchy formula to decompose
\begin{equation*}
\bw {r+1}V\otimes \bw {r+1}(U\otimes V)
= \bw {r+1}V\otimes \bigoplus_{|\pi|=r+1}S_{\pi}U \otimes S_{\pi'}V.
\end{equation*}
Applying the Pieri rule to $\bw{r+1} V \otimes S_{\pi'} V$, we see that all of the summands
have more than $r+1$ parts except for $S_{\pi' + 1^{r+1}} V$. Therefore, the
decomposition of $I_{\kappa_0 \leq r}$ must consist only of the modules
$S_{\pi}U \otimes S_{\pi' + 1^{r+1}}V$, where $\pi$ is a partition of $r+1$, and
each such module can occur with multiplicity at most one.

For the second step, fix $\pi$ a partition of $r+1$.  Suppose that
$u_{1},\ldots,u_{m}$ is our ordered basis of $U$ and $v_{1},\ldots,v_{n}$ is
our ordered basis of $V$. Consider the
indecomposable basis element
\begin{multline*}
z_{\pi}=
(v_{1}\wedge \dots\wedge v_{r+1}) \otimes
 \big( (u_{1}\otimes v_{1})\wedge \dots\wedge (u_{1}\otimes v_{\pi_{1}})
\wedge (u_{2}\otimes v_{1})\wedge \dots\wedge (u_{2}\otimes v_{\pi_{2}}
) \\
\wedge \cdots
 \wedge (u_{m}\otimes v_{1})\wedge \dots\wedge (u_{m}\otimes v_{\pi_{m}}
)\big)
\end{multline*}
in 
$\bw {r+1} V \otimes \bw {r+1} \left(U\otimes V \right)$.  We claim that
$z_{\pi}$ is in $S_{\pi}U \otimes S_{\pi' +1^{r+1}}V$, and, in fact is a non-zero highest
weight vector in that representation.
The vector $z_\pi$ is non-zero because $z_\pi$ is the tensor product of two
tensors, each constructed as an
exterior product of linearly independent tensors and hence non-zero.
It is clear that $z_\pi$ has weights $\pi$ and $\pi' + 1^{r+1}$ in $U$ and~$V$
respectively, with respect to our chosen bases. Moreover,
replacing $v_i$ by~$v_j$ or $u_i$ by~$u_j$, with $j <
i$ in either case, would result in a repeated term in the exterior product, and
thus any raising operator would send $z_\pi$ to zero, so $z_{\pi}$ is a highest weight vector.

Let $M_{\pi}$ be the submatrix of the block matrix $\psi_{0,x}^T=\begin{pmatrix} A_1 & \cdots & A_m
\end{pmatrix}$ defined by selecting the first $r+1$ rows and the
first $\pi_i$ columns of the $i$th block for each $i \leq n$.
Then the map $\Phi_0$ sends $z_{\pi}$ to the determinant of $M_{\pi}$.
For appropriate choices for $A_{i}$, we can make $M_{\pi}$ equal the identity
matrix, and therefore $\Phi(z_{\pi})=\det(M_{\pi})$ is nonzero.
\end{proof}

In the case $\dim(U)=3$, we similarly produce a formula for the decomposition of
the modules generating $I_{\kappa_{1}\leq 2r}$ in
Proposition~\ref{prop:kappa1sym:decomp}.

\begin{remark}
Taken together, PropoPropositionssitions~\ref{prop:kappa0:decomp} and \ref{prop:m-equals-1} provide a complete Schur module description of the generators of the ideal of any secant variety of the $\mathbb P^1\times \mathbb P^{n-1}$ embedded by~$\mathcal O(1,2)$.
Similarly, Propositions~\ref{prop:kappa0:decomp} and
\ref{prop:kappa1sym:decomp}, together with Theorem~\ref{thm:main}, provide a complete Schur module description of the generators of the ideal of the $r$th  secant variety of $\mathbb P^2\times \mathbb P^{n-1}$ embedded by $\mathcal O(1,2)$ for $r$ at most~$5$.
\end{remark}

\begin{prop}\label{prop:kappa1sym:decomp}
Suppose that $\dim(U)$ is~$3$.
As a Schur module, we have the following decomposition of the generators $I_{\kappa_1\leq 2r}$ into irreducible representations of $\GL(U)\times \GL(V)$
\[(I_{\kappa_{1}\leq 2r})_{r+1} 
= \bigoplus_{|\pi|=r+1} S_{\pi}U \otimes S_{(3)^{r+1} - \pi'}V,\]
where $\pi'$ is the conjugate partition to $\pi$. 
In order for the summand to be non-zero, $\pi$ must have at most $3$~parts, and
$\pi_3$ must be at least $m - r+1$ (if the latter is positive).
\end{prop}

\begin{proof}
We consider the Pfaffians of a matrix representing the map
$\psi_{1,x}\colon V\otimes  U^* \to V^*\otimes \bw {2} U^*$.
In order to view $\psi_{1,x}$ as a skew-symmetric transformation,
we identify $\bw 2 U^*$ in the target with~$U \otimes \bw 3 U^*$. Then,
we can view $\psi_{1,x}$ as a skew-symmetric form on~$V \otimes U^*$,
taking values in $\bw 3 U^*$. Equivalently, a choice of a nonzero element
in~$\bw 3 U$ gives a $\CC$-valued skew-symmetric form.

%we choose an
%orientation $\bw 3 U\cong \CC$, as this orientation then allows us to identify
%$(V^*\otimes \bw {2} U^*)^*$ with $V\otimes \bw 1 U^*$.  By choosing bases for
%$U$ and~$V$, we may then view $\psi_{1,x}$ as skew-symmetric matrix of linear forms on $S^\bullet(U\otimes S^2V)$.

The remainder of our proof essentially follows the same two steps as the proof
of Proposition~\ref{prop:kappa0:decomp}. The space of $(2r+2) \times
(2r+2)$-Pfaffians of a skew-symmetric form on $V \otimes U^*$ is isomorphic to
$\bw {2r+2} (V \otimes U^*)$. Therefore,  similar to
Lemma~\ref{lem:kapp2schur},
$(I_{\kappa_{1}\leq 2r})_{r+1}$ is the image of the map
\[
\Phi_1\colon\ \bw {2r+2}\left( V\otimes U^* \right) 
\otimes (\bw 3 U)^{r+1}
\to %S^{r+1}\left(U\otimes S^2(V)\right)\otimes \left( \bw 3 U\right)^{r+1}\cong
S^{r+1}\left(U\otimes S^2(V)\right),
\]
which sends an indecomposable basis element to the
corresponding Pfaffian.
Note that the power of $r+1$ in $(\bw 3 U)^{r+1}$ is because the Pfaffian has
degree $r+1$.
First, we show that only modules of the form $S_{\pi}U
\otimes S_{(3)^{r+1}-\pi'} V$ can arise as a representation in both the source and target
of~$\Phi_1$.  Second, we consider $\Phi_1$ restricted to $S_{\pi}U \otimes
S_{(3)^{r+1}-\pi'} V$ and we produce a Pfaffian in the image and explicitly show
that it is non-zero. By Schur's Lemma, this will show that every such
representation actually arises in the decomposition of $(I_{\kappa_{1}\leq 2r})_{r+1}$.

For the first step, we use the skew Cauchy formula to decompose the source
of~$\Phi_1$ as
\begin{equation*}
\bw {2r+2}(U^{*} \otimes V)\otimes (\bw {3}U)^{r+1}  = 
\bigoplus_{\lvert\lambda\rvert = 2r+2}
S_{\lambda}U^{*}  \otimes S_{\lambda'}V 
 \otimes (\bw {3}U)^{r+1}
= \bigoplus_{\lvert\lambda\rvert = 2r+2}
S_{(r+1)^{3}-\lambda}U \otimes
S_{\lambda'}V
,\end{equation*}
where we have used the duality formula from Remark~\ref{rmk:repdual} for the second
equality.  Every module in the source of $\Phi_1$ is thus
of the form $S_{(r+1)^{3}-\lambda}U \otimes S_{\lambda'}V$, where $\lambda$ is a partition of $2r+2$.
We make the substitution $\lambda =
(r+1)^{3}-\pi$ to arrive at the expression in the statement of the proposition.

For the second step, we explicitly produce a non-zero Pfaffian of $\psi_{1,x}$ in the
image of $\Phi_{1}$ restricted to $S_{\pi'}U \otimes S_{(r+1)^{3} - \pi}V$, and
thus confirm that every module of the form $S_{\pi'}U \otimes S_{(r+1)^{3} -
\pi}V$ (for appropriate $\pi$) occurs in the decomposition of
$(I_{\kappa_{1}\leq 2r})_{r+1}$.  Suppose that $u_1,u_2, u_3$ is our ordered
basis for~$U$ and $v_1, \dots, v_n$ is our ordered basis for~$V$. 
Let $\pi=(\pi_1,\pi_2,\pi_3)$ be a a partition of $r+1$ with no more than three
parts, and let $\lambda = (r+1)^{3} - \pi = 
(r+1-\pi_{3},r+1-\pi_{2},r+1-\pi_{1})$, as before.
Consider the element
\begin{multline*}
z_{\pi} = \big(({u}_{1}^*\otimes v_{1})\wedge \ldots \wedge (u^*_{1}\otimes
v_{\lambda_{3}})
\wedge
(u^*_{2}\otimes v_{1})\wedge \ldots \wedge (u^*_{2}\otimes
v_{\lambda_{2}})
\wedge
({u}_{3}^*\otimes v_{1})\wedge \ldots \wedge (u_{3}^*\otimes
v_{\lambda_{1}})\big) \\
\otimes (u_1 \wedge u_2 \wedge u_3)^{\otimes r+1}
\end{multline*}
in $\bw {2r+2}(U^* \otimes V) \otimes (\bw 3 U)^{r+1}$, where the
$u^*_{i}$ form the dual basis to the~$u_i$.  Note that $z_\pi$ is non-zero
since the vectors in each exterior product are linearly independent.  Now, we will
show that $z_\pi$ is a highest weight vector in $S_\pi U \otimes S_{\lambda'}V$.

First, we claim that $z_{\pi}$ has weight $(\pi,\lambda')$. By counting the
occurrences of $v_i$ in~$z_\pi$, it is clear that the weight
in the $V$-factor is $\lambda'$.  For the $U$ factor, we note that the weight
of~$u^*_i$ is the negative of that of $u_i$, so that $z_\pi$  has weight
$(r+1 - \lambda_3, r+1 - \lambda_2, r+1 - \lambda_1)$, which is equal to~$\pi$.

Second, we must show that any raising operator will send $z_{\pi}$ to zero,
which will imply that $z_{\pi}$ is a highest weight vector.
In the $V$ factor, sending $v_{i}$ to $v_{j}$ with $j<i$ would force a repeated
vector in the exterior product. Likewise, a raising operator applied to the $U$
factor would send $u^*_{i}$ to $u^*_{j}$ with $j > i$, which would,
again, create a repeated factor in the exterior product. 

Finally, we check that $\Phi_1(z_{\pi})\ne 0$. Let $M_{\pi}$ be the
principal submatrix obtained from $\psi_{1,x}$ by selecting the rows and columns
with indices $\{1, \dots,\lambda_{3}, n+1, \dots, n+\lambda_{2}, 2n+1, \dots,
2n+\lambda_{1}\}$.  Then $\Phi_1(z_{\pi})$ equals the Pfaffian of $M_{\pi}$.  
To check that the Pfaffian of $M_{\pi}$ is nonzero, it suffices to 
produce a
specialization of $M_{\pi}$ which has full rank.  Note that if
$B_i$ is the appropriate submatrix from the upper-left corner of $A_i$, then
$M_{\pi}$ has the
following shape
\begin{equation*}
M_{\pi}=
\bordermatrix{
 &\lambda_3 & \lambda_2&\lambda_1\cr
 \lambda_3&0&B_3&-B_2\cr
 \lambda_2&-B_3^t&0&B_1\cr
 \lambda_1&B_2^t&-B_1^t&0
}.
\end{equation*}
We have $\lambda_{1} = \pi_{1}+\pi_{2}$, $\lambda_{2} = \pi_{1}+\pi_{3}$,
$\lambda_{3} = \pi_{2}+\pi_{3}$.
If we specialize the $A_i$ such that the $B_i$ are as follows
\begin{equation*}
B_1=
\bordermatrix{
 &\pi_{2} &\pi_1\cr
\pi_{1}&0&\Id_{\pi_1}\cr
 \pi_3&0&0
}, \quad
B_2=
\bordermatrix{
 &\pi_2 & \pi_{1}\cr
\pi_{3} &0&0\cr
\pi_2&-\Id_{\pi_2}&0
},
\text{ and }
B_3=
\bordermatrix{
 &\pi_1 & \pi_{3}\cr
\pi_3&0&\Id_{\pi_3}\cr
\pi_{2} &0&0
},
\end{equation*}
then the specialization of $M_{\pi}$ has full rank, since it is the standard
block skew-symmetric matrix.
\end{proof}

\begin{example}
Consider the case $n=4$ and $c_1=10$.  Since $\psi_{1,x}$ is a 
skew-symmetric
$12\times 12$ matrix, we expect the ideal $I_{\kappa_1\leq 10}$ to be a
principal ideal, generated by a polynomial in $S^6(U\otimes S^2(V))$.  Applying
Proposition~\ref{prop:kappa1sym:decomp}, we must have a sum over partitions~$\pi$
of~$6$ such that $(3)^6 - \pi'$ has at most $4$ parts. This forces $\pi'$ to equal
$(3,3)$, and thus $\pi=(2,2,2)$. The generators of $I_{\kappa_1 \leq
10}$ are therefore equal to the $1$-dimensional representation
$S_{2,2,2}(U) \otimes S_{3,3,3,3}(V)$, corresponding to the Pfaffian of $\psi_{1,x}$.
\end{example}

\begin{example}\label{ex:363}
Consider the case $n=4$ and $c=(3,6,3)$, which we revisit in
Example~\ref{ex:23}.  Propositions~\ref{prop:kappa0:decomp}
and~\ref{prop:kappa1sym:decomp} give us the decompositions:
\begin{align*}
(I_{\kappa_0\leq 3})_{4} &= S_{2,2} S_{3,3,1,1}
\oplus  \mathbf{S_{2,1,1} S_{4,2,1,1} }
\oplus  S_{3,1} S_{3,2,2,1} 
\oplus  S_{4} S_{2,2,2,2} 
,
\\
(I_{\kappa_1\leq 6})_{4} &= S_{2,2} S_{3,3,1,1} 
\oplus  \mathbf{S_{2,1,1} S_{3,3,2}}
\oplus  S_{3,1} S_{3,2,2,1} 
\oplus  S_{4} S_{2,2,2,2} 
.\end{align*}
Both modules are $495$-dimensional and consist of quartic polynomials. The ideal
$I_{\kappa\leq (3,6,3)}$, which equals $I_{\kappa_6\leq 6}+I_{\kappa_0\leq 3}$
by definition, is generated by the $630$-dimensional space of quartics obtained
by taking the sum of the above decompositions.  Notice that, due to the highlighted modules in the above decompositions, neither $I_{\kappa_0\leq 3}$ nor $I_{\kappa_1\leq 6}$ belongs to the other. 
In particular, the $4\times 4$-minor formed by taking columns $1$, $2$, $5$,
and~$9$ from the flattening $\psi_{0,x}$, which is the transpose of~(\ref{eq:flattening}), is not in the ideal of
Pfaffians.  On the other hand, the Pfaffian formed by taking the rows and columns
of~(\ref{eq:skew-sym}) with indices $1$, $2$, $5$, $6$, $7$, $9$, $10$, and~$11$
is not contained in the ideal of the minors.
\end{example}

%\begin{remark}\label{prop:01duals}
%Suppose $r+1 = \dim(V)$. Let $\Omega\colon U\otimes S^{2}V \rightarrow U\otimes
%S^{2}V^{*}$ be the isomorphism induced by a choice of non-degenerate element in
%$S^{2}V$. \dustin{This choice is not $\SL(V)$-equivariant. (which is why
%$\operatorname{SO}(V)$ is a proper subgroup of $\SL(V)$). Can we just drop this
%remark?} Then $\Omega((I_{\kappa_{0}\leq r})_{r+1}) = (I_{\kappa_{1}\leq
%2r})_{r+1}$ This follows immediately from comparing the expression for
%$(I_{\kappa_{0}\leq r})_{r+1}$ given in Proposition \ref{prop:kappa0:decomp} and
%the expression for $(I_{\kappa_{1}\leq 2r})_{r+1}$ given in
%Proposition~\ref{prop:kappa1sym:decomp}
%\end{remark}

Notice that the formulas in Propositions~\ref{prop:kappa0:decomp}
and~\ref{prop:kappa1sym:decomp} are multiplicity free, unlike, for example, the ideal
generators computed in Example~\ref{ex:non-sym-333}.

%%%%%%%%%%%%%%%%%%%%%%
%%%%%%%%%%%%%%%%%%%%%%
\section{Subspace varieties of partially symmetric tensors}\label{sec:subspace}
%%%%%%%%%%%%%%%%%%%%%%
%%%%%%%%%%%%%%%%%%%%%%

We next give a geometric interpretation for the varieties $\Sigma_{\kappa_0 \leq
r}$. These are the partially symmetric analogues of the subspace varieties defined
in~\cite[Defn.~1]{lw}; Proposition~\ref{prop:subspace} forms an analogue
to~\cite[Thm.~3.1]{lw}.

\begin{defn}\label{defn:subspace}
The \defi{subspace variety} $\Sub_{m',n'}$ is the variety of tensors $x \in 
(U^{*}\otimes S^{2}V^{*})$ such that there exist vector spaces $\widetilde U^*
\subset U^*$ and $\widetilde V^* \subset V^*$ of dimensions $m'$ and $n'$
respectively with $x \in (\widetilde U^* \otimes S^{2} \widetilde V^*) \subset
(U^* \otimes S^2  V^*)$.
\end{defn}

\begin{remark}\label{rmk:weyman}
The variety $\Sub_{m',n'}$ has a nice desingularization, 
analogous to that used to prove the results in~\cite[\S3]{lw}.  
Consider the product of Grassmannians $\Gr(m',U^{*})\times \Gr(n',V^{*})$,
and let $E$ be the total space of
the vector bundle $\mathcal{R}_{U}\otimes S^{2}\mathcal{R}_{V}$, where
$\mathcal{R}_{U}$ and $\mathcal{R}_{V}$ are the tautological subbundles over
$\Gr(m',U^{*})$ and $\Gr(n',V^{*})$, respectively.  Then there is a natural map
$\pi\colon E\to \Sub_{m',n'}$, which desingularizes $\Sub_{m', n'}$.  Moreover,
one can verify that Weyman's geometric
technique can be applied in this situation~\cite[\S5]{weyman-book}.  In fact, a
straightforward adaptation of the argument in~\cite[Thm.~3.1]{lw} implies that
$\Sub_{m',n'}$ is normal with rational singularities.
\end{remark}

We next directly calculate the generators of the ideal of the subspace variety
when $m' = m$, which is the case we need.

\begin{prop}\label{prop:subspace}
The defining ideal of $\Sub_{m,n'}$ equals
$I_{\kappa_0\leq n'}$.
\end{prop}

The following lemma plays a crucial technical role in the proof of both 
Proposition~\ref{prop:subspace} and Theorem~\ref{thm:main}, as it provides
a criterion for determining the reducedness of some of the ideals that we
are studying.

\begin{lemma}\label{lem:subspace-projection}
Let $Z$ be a $\GL(U) \times \GL(V)$-invariant reduced subscheme of
the desingularization~$E$ from Remark~\ref{rmk:weyman}.  
Suppose that $I$ is an
invariant ideal in $S^\bullet(U \otimes S^2 V)$, which contains the ideal of
$\Sub_{m',n'}$, and whose pullback to $E$ defines~$Z$. Then $I$ is the ideal
of~$\pi(Z)$.  In particular, $I$ is a radical ideal.
\end{lemma}

\begin{proof}
Let $J\subseteq S^\bullet(U \otimes S^2 V)$ be the defining ideal of $\pi(Z)$.
Recall that $q\colon E\to \Gr(m',U^*\times \Gr(n',V^*)$ is the total space of a vector bundle, as a defined in Remark~\ref{rmk:weyman}.  Our set-up is the following commutative diagram:
\[
\xymatrix{
Z\ar@{^(->}[r]\ar[d] &E\ar[d]^\pi&\\
\pi(Z)\ar@{^(->}[r]& \Sub_{m',n'}\ar@{^(->}[r]&\left( U^*\otimes S^2V^*\right).
}
\]
The hypothesis that the pullback of $I$ defines $Z$ in $E$ guarantees $I\subseteq J$.
We thus need to show the reverse inclusion.

A point $P \in \Gr(m', U^*) \times \Gr(n', V^*)$
corresponds to vector subspaces $\widetilde U^* \subset U^*$ and $\widetilde V^*
\subset V^*$, and this induces 
a surjection of rings $\phi_P \colon S^\bullet(U
\otimes S^2 V) \rightarrow S^\bullet(\widetilde U \otimes S^2 \widetilde V)$.  The fiber $q^{-1}(P)\subseteq E$ is isomorphic
to the affine space $(\widetilde U^* \otimes S^2 \widetilde V^*)$, and we define
$\widetilde{Z}_P:=Z\cap q^{-1}(P)$.

We claim that a polynomial $f\in S^\bullet(U \otimes S^2 V)$ belongs to $J$ if
and only if $\phi_P(f)$ vanishes on $\pi(\widetilde{Z}_P)$ for every choice
of~$P$.  The ``only if'' direction of the claim is straightforward.  For the
``if'' direction, we first note that, since $Z$ is assumed to be reduced, the
condition $f\in J$ is equivalent to the condition that the pullback of~$f$
vanishes on every point $y\in \pi(Z)$.  This is in turn equivalent to the
condition $f$ vanishes on each point of $Z$, which is implied by the hypothesis
that $\phi_P(f)$ vanishes on $\pi(\widetilde{Z}_P)$ for each $P$.

In fact, since $\GL(U)\times \GL(V)$ acts transitively on $\Gr(m', U^*) \times \Gr(n', V^*)$, we conclude that $f$ vanishes on $\pi(Z)$
if and only if $\phi_P(g \cdot f)$ vanishes on $\pi(\widetilde{Z})$ for \emph{any}
fixed choice
of~$P$ and all $g \in \GL(U) \times \GL(V)$.  Therefore, $J$
is the sum of all irreducible Schur submodules $M$ of $S^\bullet (U \otimes S^2 V)$ such that
$\phi_P(M)$ vanishes on $\widetilde{Z}_P$.
For the rest of the proof we fix $\widetilde U^*$ and
$\widetilde V^*$ and denote the induced map~$\phi_P$ by~$\phi$ and
$\widetilde{Z}_P$ by~$\widetilde{Z}$.

To show that $J\subseteq I$, let $M$ be an irreducible Schur
submodule of the ideal $J$; we want to
show that $M$ is contained in our given ideal~$I$.  If $M$ is isomorphic to $S_\mu U \otimes
S_\nu V$, then the construction of Schur modules implies that $\phi(M)$ is
isomorphic as a $\GL(\widetilde U) \times \GL(\widetilde V)$-representation to $S_\mu
\widetilde U \otimes S_\nu \widetilde V$, which is either trivial or an irreducible
representation.  We know that $\phi(M)$ vanishes on $\widetilde{Z}$
and thus, since $I$ pulls back to the defining ideal of $Z$, it follows that
 $\phi(M)$ is contained in $\phi(I)$. There is thus an irreducible Schur submodule $N \subset I$ such that
$\phi(N) = \phi(M)$, and hence $N$ is isomorphic to $S_\mu U \otimes S_\nu V$. If
$N$ equals $M$, then we are done. Otherwise, $\phi$ sends the submodule
$N+M$, spanned by two copies of $S_\mu U \otimes S_\nu V$,
to the submodule $\phi(M)$, which is a single copy of $S_\mu \widetilde U \otimes
S_\nu \widetilde V$.
Thus, some subrepresentation $L$ of $N+M$ is sent to zero by $\phi$. Since $L$
is a representation in the kernel of $\phi$, $L$ belongs to the ideal of $\Sub_{m', n'}$, which
is contained in~$I$ by
assumption.
It follows that $I$ contains the span of $N$ and $L$,
and hence $I$ contains $M$.  We conclude that $I=J$ as desired.
\end{proof}

\begin{proof}[Proof of Proposition~\ref{prop:subspace}]
First, we prove the claim set-theoretically. 
The $(n' + 1) \times (n' + 1)$
minors of $\psi_{0,x}$ vanish if and only if the map has rank at most $n'$.
By linear algebra, this is equivalent to the existence of a change of basis in
which $\psi_{0,x}$ uses only the first $n'$ rows, which is the definition of
$\Sub_{m, n'}$.

Second, we show that $I_{\kappa_0 \leq n'}$ is radical in the case when $n' = n-
1$. Note that $\Sub_{m,n-1}$ has dimension $m\binom{n}{2}+n$,
and thus $\Sub_{m,n-1}$ and~$\Sigma_{\kappa_0 \leq n-1}$ have codimension
$mn - n +1$. This is the same as the codimension of the maximal
minors of a generic $n \times mn$ matrix, so $I_{\kappa_0 \leq n-1}$ is
Cohen-Macaulay by~\cite[Thm.\ 18.18]{eisenbud}, and it suffices to show that
the affine cone over $\Sigma_{\kappa_0 \leq n-1}$ is reduced at some point. Consider a neighborhood
of the point $u_1\otimes v_1^2 + \cdots + u_1\otimes v_{n-1}^2$. In coordinates around this point, $I_{\kappa_0 \leq n-1}$ consists
of the maximal minors of the $n \times mn$ matrix:
\begin{equation*}\begin{bmatrix}
1 + x_{1,1,1} &  \cdots & x_{1,1,n-1} & x_{1,1,n} &
x_{2,1,1} & \cdots & x_{2,1,n} & \cdots & x_{m,1,n} \\
\vdots & \ddots & \vdots & \vdots & \vdots & & \vdots & \cdots & \vdots \\
x_{1,1,n-1} & \cdots & 1 + x_{1,n-1,n-1} & x_{1,n-1,n} & 
\vdots & & \vdots & \cdots & \vdots \\
x_{1,1,n} & \cdots & x_{1,n-1,n} & x_{1,n,n} &
x_{2,1,n} & \cdots & x_{2,n,n} & \cdots & x_{m,n,n}
\end{bmatrix}.\end{equation*}
The $mn - n +1$ minors which use the first $n-1$ columns form part of a regular
sequence, and thus the affine cone over $\Sigma_{\kappa_0 \leq n-1}$ is reduced in a neighborhood of
this point. Since $I_{\kappa_0 \leq n-1}$ is a Cohen-Macaulay ideal, it
follows that $\Sigma_{\kappa_0\leq n-1}$ is everywhere reduced.

Third, we show that $I_{\kappa_0 \leq n'}$ defines $\Sub_{m, n'}$ for arbitrary
$n'$.  By reverse induction on $n'$, we assume that $I_{\kappa_0 \leq n'}$
equals the ideal of $\Sub_{m, n'}$, and we seek to show equality for $n'-1$.  We
will apply Lemma~\ref{lem:subspace-projection}, where $E$ is the vector bundle
over $\Gr(m, U^*) \times \Gr(n', V^*) = \Gr(n', V^*)$ desingularizing
$\Sub_{m,n'}$ as in Remark~\ref{rmk:weyman}. Note
that, by cofactor expansion, $I_{\kappa_0\leq n'-1}$ contains $I_{\kappa_0 \leq
n'}$, which is the ideal of $\Sub_{m,n'}$ by the inductive hypothesis.  We
describe $Z$, which is defined by the pullback of $I_{\kappa_0 \leq n' - 1}$,
on a local trivialization $(U^* \otimes S^2 \widetilde V^*) \times Y$ of the vector bundle~$E$, where
$\widetilde V$ is $n'$-dimensional and $Y$ is an open subset of $\Gr(n', V^*)$.
The pullbacks of the $(n' - 1) \times (n' - 1)$ minors of $\psi_{0,x}$ do not involve the
base~$Y$, and are the $\kappa_0
\leq n' -1 $ equations applied to $U^* \otimes S^2 \widetilde V^*$.  These are maximal minors
of the matrix $\psi_{0,x}$ for $U^* \otimes S^2 \widetilde V^*$, and hence they
define a reduced subscheme of $U^* \otimes S^2 \widetilde V^*$  by 
the previous paragraph. In the local
trivialization, their scheme is the product of this reduced scheme with $Y$, so
the preimage of $\Sigma_{\kappa_0 \leq n'}$ in~$E$ is reduced.  We may thus apply
Lemma~\ref{lem:subspace-projection} and conclude that $I_{\kappa_0 \leq n'-1}$ is reduced.
\end{proof}

\begin{remark}
When $m' < m$, the ideal of $\Sub_{m',n'}$ is similarly generated by the sum of
$I_{\kappa_0 \leq n'}$ and the irreducible modules in $\bw{m'+1} U \otimes
\bw{m'+1}(S^2V)$. A decomposition of the latter space,
in somewhat different notation, can be found at~\cite[p.\ 47]{Macdonald}.
\end{remark}

%%%%%%%%%%%%%%%%%

%%%%%%%%%%%%%%%
%%%%%%%%%%%%%%%
\section{Secant varieties of \texorpdfstring{$\PP^2 \times \PP^{n-1}$}{P2 x
Pn-1} embedded by $\Osh(1,2)$}
\label{sec:m-3}
%%%%%%%%%%%%%%%
%%%%%%%%%%%%%%%
In this section, we prove the main result of our paper, which is to show that the equations
given in Definition~\ref{def:ideal-non-sym} generate the defining ideal of the $r$th secant variety of $\PP^2 \times \PP^{n-1}$ embedded by $\Osh(1,2)$ when $r\leq 5$.

We first consider a simpler case: the secant varieties of $\PP^1\times \PP^{n-1}$
embedded by $\Osh(1,2)$.  All such secant varieties are defined by
$\kappa$-equations, which, in this case,
are simply the minors of flattenings.
The analogous statement for non-symmetric matrices appears as
Theorem~1.1 in~\cite{lw}.  However, we know of no proof in the literature
for the case of partially symmetric tensors, so we provide one below.
%Moreover, for partially symmetric tensors, the
%set-theoretic version follows easily from results in the theory of matrix
%pencils.  
%However, since we know of no proof of the ideal-theoretic statement 

\begin{defn}
For a variety $X\subseteq \PP^N$ we denote the affine cone of~$X$ in $\mathbb A^{N+1}$ by~$\widehat{X}$.
\end{defn}

\begin{prop}\label{prop:m-equals-1}
Suppose $m= 2$, and let $Y\subseteq \mathbb P(U^*\otimes S^2 V^*)$ be the image
of $\PP(U^*) \times \PP(V^*)$ under the embedding by $\Osh(1,2)$.  For any $r>1$, and any $n$, the secant variety $\sigma_r(Y)$ is defined ideal-theoretically by the ideal $I_{\kappa_0\leq r}$.  
\end{prop}
\begin{proof}
We have
$
\widehat{\sigma_r}(Y) \subseteq \widehat{\Sigma}_{\kappa_0\leq r}=\Sub_{2,r},
$
where the inclusion follows from Proposition~\ref{prop:vanish:semi} and the equality follows from Proposition~\ref{prop:subspace}.  Since $\Sub_{2,r}$ is integral, it suffices to prove that $\widehat{\sigma_r}(Y)$ and $\Sub_{2,r}$ have the same dimension.  By~\cite[Cor.\ 1.4(i)]{abo}, the former has the expected dimension $rn + r $.  From the definition of $\Sub_{2,r}$, we can compute its
dimension to be $r(n-r) + 2\binom{r+1}{2} = rn + r $.
\end{proof}

For the remainder of this section, we restrict to the case when $\dim U^*
= 3$, which is the next partially symmetric case.
We let $\dim V^*=n$, and we consider
partially symmetric tensors $x \in U^* \otimes S^2 V^*$. We fix $\mathbb P^{N}:=\mathbb P(U^*\otimes S^{2}(V^*))$ and we
let $X \subset \PP^{N}$ denote the embedding of $\PP(U^*) \times \PP(V^*)$ by
$\Osh(1,2)$.
Let $S:=S^{\bullet}(U\otimes S^{2}(V))$ be the homogeneous coordinate ring
of~$\mathbb P^N$, which contains the ideals $I_{\kappa_j \leq c_j}$ and~$I_{\kappa
\leq c}$ as in Definition~\ref{def:ideal-sym}.

\begin{thm}\label{thm:main}
For $r\leq 5$, the defining ideal of the variety $\sigma_r(X)$ is 
$I_{\kappa\leq (r,2r,r)}$.
\end{thm}

Our method of proof is as follows.  When $n$ equals $r$, we relate the ideal
$I_{\kappa_1 \leq 2r}$ to the ideal of commuting symmetric matrices.  This is a variant of an idea which has appeared in several instances previously~\cite{strassen, Ottaviani07,shuchao}.
This
relation
only holds away from a certain closed subvariety of $\PP^N$, and in order to extend
to all of $\PP^N$, we need a bound on the dimension of this variety.  Such a bound is given in
in~\cite[\S5]{ev}, and only holds for $r \leq 5$. Finally, we reduce the general
case to the case of $n=r$, using Lemma~\ref{lem:subspace-projection}.

Before the proof, we examine the secant varieties of $\PP^2\times \PP^3$ in more
detail.
\begin{example}\label{ex:23}
%[The secant varieties of $\PP^2\times \PP^3$ embedded by $\mathcal O(1,2)$]\label{ex:23}
Let $X\subseteq \PP^{29}$ be the image of $\PP^2\times \PP^3$ embedded by $\mathcal O(1,2)$.  The defining ideal of $\sigma_5(X)$ was previously known. The secant variety $\sigma_5(X)$ is deficient, and is in fact a hypersurface in $\PP^{29}$.  This hypersurface is defined by the Pfaffian of $\psi_{1,x}$~\cite[Thm.\ 4.1]{Ottaviani07}.

In the non-symmetric case, ~\cite[Thm.\ 1.1]{lw} illustrates that the defining ideal for the second secant variety is generated by the $3\times 3$ minors of the various flattenings.  This suggests that a similar result holds in the partially symmetric case, although we know of no explicit reference for such a result.  Nevertheless, in the situation of this example, a direct computation with~\cite{M2} confirms that the defining ideal of $\sigma_2(X)$ is indeed generated by the $3\times 3$ minors of the flattening $\psi_{0,x}$ and by the $3\times 3$ minors of the other flattening of $x$, i.e.\ by considering $x$ in $\Hom(U, S^2 V^*)$. Theorem~\ref{thm:main} provides an alternate description, illustrating that the $3\times 3$ minors of $\psi_{0,x}$ and the $6\times 6$ principal Pfaffians of $\psi_{1,x}$ also generate the ideal of $\sigma_2(X)$.
%Namely, the defining ideal for $\sigma_2(X)$ is generated by the $3\times 3$ minors of the flattening $\psi_{0,x}$~\cite[Thm.\ 1.1]{lw}.   
%Curiously, the $6\times 6$ principal Pfaffians of $\psi_{1,x}$ define $\sigma_2(X)$ set-theoretically but {\em not} ideal-theoretically. \dan{We had noted this previously, though I don't the computation to back this up.}  

As far we are aware, the defining ideals for $\sigma_3(X)$ and $\sigma_4(X)$ were not previously known.  In the case of $\sigma_4(X)$, the defining ideal is given by $I_{\kappa\leq (4,8,4)}$.  Since the ideals $I_{\kappa_0\leq 4}$ and $I_{\kappa_2\leq 4}$ are trivial, this equals the ideal $I_{\kappa_1\leq 8}$.  Thus, $\sigma_4(X)$ is defined by the $10\times 10$ principal Pfaffians of $\psi_{1,x}$.

The case of $\sigma_3(X)$ is perhaps the most interesting, since this case requires minors from both $\psi_{0,x}$ and $\psi_{1,x}$ (and, unlike the case of $\sigma_2(X)$, the Pfaffians from $\psi_{1,x}$ do not arise from an alternative flattening).  Here $\sigma_3(X)$ is defined by the maximal minors of $\psi_{0,x}$ as well as the $8\times 8$ principal Pfaffians of $\psi_{1,x}$.  By Example~\ref{ex:363}, we see that neither $I_{\kappa_0\leq 3}$ nor $I_{\kappa_1\leq 6}$ is sufficient to generate the ideal of $\sigma_3(X)$.

In fact, neither $I_{\kappa_0\leq 3}$ nor $I_{\kappa_1\leq 6}$ is sufficient to define $\sigma_3(X)$ even set-theoretically.  For $I_{\kappa_0\leq 3}$, this follows from the fact that a generic element $y\in \Sigma_{\kappa_0\leq 3}$ has $\kappa_1(y)=8$.  On the other hand, one may check that if
\[
x:=\sum_{i=1}^3 u_i\otimes(v_1\otimes v_{i+1}+v_{i+1}\otimes v_1)
\in U^*\otimes S^2 V^*,
\]
then $\kappa(x)=(4,6,4)$, and hence $[x]$ belongs to $\Sigma_{\kappa_1\leq 6}$ but not to $\sigma_3(X)$.
\end{example}

\begin{remark}\label{rmk:Comon}
Let $x\in U^*\otimes S^2 V^*$ and let $r\leq 5$.  Theorem~\ref{thm:main} implies that the border rank of $x$, considered as an element of $U^*\otimes V^*\otimes V^*$, equals the partially symmetric border rank of $x$.  This is because the ideal $I_{\kappa\leq (r,2r,r)}$ is (up to radical) the restriction to $\PP(U^*\otimes S^2 V^*)$ of an ideal on $\PP(U^*\otimes V^*\otimes V^*)$ which vanishes on the $r$th secant variety of $\PP(U^*)\times \PP(V^*)\times \PP(V^*)$ (see Proposition~\ref{prop:vanish} and Definition~\ref{def:ideal-sym} above).  This can thus be viewed as evidence for a partially symmetric analogue of Comon's Conjecture~\cite[\S 5]{cglm}.
\end{remark}

\begin{defn}\label{defn:purely}
If we write $x = e_1 \otimes A_1 + e_2 \otimes A_2 + e_3 \otimes A_3$ for $e_1,
e_2, e_3$ a basis of $U$ and the $A_i$ symmetric matrices, then 
$\det(t_1 A_1 + t_2 A_2 + t_3 A_3)$ is a
polynomial in $t_1$, $t_2$, and~$t_3$. We define $P \subset \mathbb A^{N+1}$ to be the
subset of those $x$ such that this polynomial vanishes identically.
\end{defn}

\begin{remark}\label{rmk:GLU}
Note that $\mathbb A^{N+1}-P$ is exactly the $\GL(U^*) \times \GL(V^*)$-orbit of the
set $\{e_1 \otimes \Id + e_2 \otimes B + e_3 \otimes C \mid B, C \in S^2 V^*\}$.
\end{remark}

\begin{lemma}\label{lem:setP}
Let $n=r$.  Then $\widehat{\Sigma}_{\kappa\leq (r,2r,r)}-P$ is an irreducible locus of
codimension at least $\binom{r}{2}$ on $\mathbb A^{N+1}-P$.
\end{lemma}

In fact, the codimension is exactly $\binom{r}{2}$, as will be shown in the proof of
Lemma~\ref{lem:n-equals-r}.

\begin{proof}
Since $n=r$, and $\kappa_0 = \kappa_2$ are always at most $n$, we have that
$\Sigma_{\kappa_1\leq r}=\Sigma_{\kappa\leq (r,2r,r)}$.
For convenience, we denote this scheme $\Sigma$, and we seek to show that
$\widehat{\Sigma} - {P}$ is irreducible and of codimension~$\binom{r}{2}$.  

We let $\mathcal W\subseteq \mathbb A^{N+1}$ be the set $\{e_1 \otimes \Id + e_2 \otimes B + e_3 \otimes C
\mid B, C \in S^2 V^*\}$ as in Remark~\ref{rmk:GLU}, and we identify points
in~$\mathcal W$ with pairs of symmetric matrices $(B, C)$.  Let $Z\subseteq
{\mathcal W}$ be the subscheme defined by the equations $[B,C]=0$.
By~\cite[Thm.\ 3.1]{bpv}, $Z$, known as the variety of commuting symmetric matrices,  is an integral subscheme of codimension $\binom{r}{2}$ in $\mathcal W$.

We claim that $\widehat{\Sigma} -{P}$ is irreducible. To see this, we note the following
equivalence of matrices under elementary row and column operations:
\begin{equation*}
\begin{bmatrix}
0 & \Id & -B \\
-\Id & 0 & C \\
B & -C & 0
\end{bmatrix} \sim
\begin{bmatrix}
0 & \Id & 0 \\
-\Id & 0 & 0 \\
0 & 0 & BC - CB
\end{bmatrix}.
\end{equation*}
Therefore, the scheme-theoretic intersection of $\widehat{\Sigma}$ with $\mathcal W$ is
exactly $Z$, the variety of commuting symmetric matrices.
By Remark~\ref{rmk:GLU} and the fact that $\kappa_1$ is
 $\GL(U^*) \times \GL(V^*)$-invariant, we see that $\widehat{\Sigma} - P$ is
exactly the $\GL(U^*)\times \GL(V^*)$ orbit of the irreducible variety~$Z$, and
therefore irreducible.

Finally, since $Z =
\mathcal W \cap \widehat{\Sigma}$, the codimension of $\widehat{\Sigma} -{P}$ in $\mathbb A^{N+1}$ is at least
the codimension of $Z$ in~$\mathcal W$, which is $\binom{r}{2}$.
\end{proof}

The following result is contained in~\cite[Proof of Cor.\ 5.6]{ev}.

\begin{lemma}\label{lem:dim-P}
If $n = r  \leq 5$, then the codimension of~$P$ in~$\mathbb A^{N+1}$ is strictly greater
than $\binom{n}{2}$.
\end{lemma}

\begin{lemma}\label{lem:n-equals-r}
Let $n = r \leq 5$. Then $\sigma_r(X)$ is defined scheme-theoretically by
$I_{\kappa_1 \leq 2r} = I_{\kappa \leq (r,2r,r)}$. Moreover, the ring
$S/I_{\kappa_1 \leq 2r}$ is Gorenstein, i.e.\ $\sigma_r(X)$ is
arithmetically Gorenstein.
\end{lemma}

\begin{proof}
The ideal of the principal $(2r+2)\times (2r+2)$-Pfaffians of a generic
skew-symmetric matrix is a
Gorenstein ideal of codimension $\binom{r}{2}$~\cite[Thm.\ 17]{kleppe-laksov-pfaffians}.
Our ideal $I_{\kappa_1 \leq 2r}$ is a linear specialization of this ideal, and
by Lemmas~\ref{lem:setP} and~\ref{lem:dim-P}, it must be irreducible and have
the same codimension.  Therefore, the linear specialization is defined by a
regular sequence, so $\Sigma_{\kappa_1 \leq 2r}$ is also arithmetically Gorenstein and
irreducible.

Hence, $\widehat{\Sigma}_{\kappa_1 \leq 2r}$ is either reduced or
everywhere non-reduced.  As in the proof of Lemma~\ref{lem:setP}, let $\mathcal W
\subseteq \mathbb A^{N+1}$ be the linear space defined by $A_1=\Id$, and consider the
scheme-theoretic
intersection $\widehat{\Sigma}_{\kappa_1 \leq 2r}\cap \mathcal W$.  Again, the codimension of
$\widehat{\Sigma}_{(r,2r,r)} \cap \Lambda$ in $\mathcal W$ is $\binom{r}{2}$, so 
the generators of the ideal of $\Lambda$ form a regular sequence on the local
ring of any point of $\widehat{\Sigma}_{(r,2r,r)}$ contained in $\mathcal W$.  The intersection is
isomorphic to the variety of commuting symmetric matrices from the proof of
Lemma~\ref{lem:setP}, which is reduced.  This implies that $\widehat{\Sigma}_{\kappa\leq
(r,2r,r)}$ is reduced as well, and thus that $\Sigma_{\kappa\leq (r,2r,r)}$ is reduced.
\end{proof}

When $r>5$ we have a partial result.  Let $J_P\subseteq S^\bullet(U\otimes S^2(V))$ be the ideal defining~$P$.
\begin{cor}\label{cor:saturated}
For any $r$, the variety $\sigma_r(X)$ is defined by the prime ideal $(I_{\kappa\leq(r,2r,r)}:J_P^\infty)$.
\end{cor}

Note that computing the saturations as in Corollary~\ref{cor:saturated} can
be non-trivial.

\begin{proof}[Proof of Theorem \ref{thm:main}]
Lemma~\ref{lem:n-equals-r} proves the theorem in the case when $n = r$
and so we just need to extend this result to the
cases when $n \neq r$. We let $N' = 3 \binom{r}{2} - 1$, so that $\PP^{N'}$
is the projective space of partially symmetric $3 \times r \times r$ tensors. We write
$X' \subset \PP^{N'}$ for the image of $\PP^2 \times \PP^{r-1}$ embedded by
$\Osh(1, 2)$.

First, suppose that $n < r$. We pick an inclusion of $V^*$ into $\CC^r$, and
also a projection from $\CC^r$ back to $V^*$. These define an inclusion $\PP^N
\rightarrow \PP^{N'}$ and a rational map
$\pi\colon \PP^{N'} \rightarrow \PP^N$ respectively. Because the projection is linear, it
commutes with taking secant varieties, so $\sigma_r(X) = \pi(\sigma_r(X'))$.
Applying Lemma~\ref{lem:n-equals-r}, we get the first equality of
\begin{equation*}
\pi(\sigma_r(X')) = \pi(\Sigma_{\kappa_1 \leq 2r})
\supset \pi(\Sigma_{\kappa_1 \leq 2r} \cap \PP^N)
= \Sigma_{\kappa_1 \leq 2r} \cap \PP^N
\supset \sigma_r(X) = \pi(\sigma_r(X')).
\end{equation*}
Note that the middle equality follows from the fact that $\pi$ is the identity
on $\PP^N$.  We conclude that $\sigma_r(X)$ is defined by $I_{\kappa_1 \leq
2r}$, which is the statement of the theorem, since the conditions on $\kappa_0$
and $\kappa_2$ are trivial when $n < r$.

Second, we want to prove the theorem when $n > r$, for which we use
Lemma~\ref{lem:subspace-projection}. We consider the subspace variety $\Sub_{3,
r} \subset \mathbb A^{N+1}$ and its desingularization~$\pi\colon E \to \Sub_{3, r}$.
By Proposition~\ref{prop:subspace}, $\Sub_{3,r}$ is the affine cone over $\Sigma_{\kappa_0 \leq r}$,
which contains $\widehat{\Sigma}_{\kappa \leq (r, 2r, r)}$. We set $Z:=\pi^{-1}(\widehat{\Sigma}_{\kappa \leq (r, 2r, r)})$.
Note that, along any fiber
$U^* \otimes S^2 \widetilde{V}^*$ of $q\colon E\to \Gr(r,V^*)$,
we have that $Z\cap (U^* \otimes S^2\widetilde{V}^*)$ is defined by
the $\kappa_1 \leq 2r$ equations
applied to $U^* \otimes S^2 \widetilde V^*$.  It follows that $Z\subseteq E$
is defined by the pullback of $I_{\kappa\leq (r,2r,r)}$.
Since $\widetilde V^*$ is $r$-dimensional,
Lemma~\ref{lem:n-equals-r} implies that $Z\cap (U \otimes S^2 \widetilde{V}^*)$
 is the cone over the $r$th secant variety of $\PP(U^*) \times \PP(\widetilde V^*)$ 
 in $U^* \otimes S^2 \widetilde V^*$.  In particular, $Z$ is reduced.
We thus have
the inclusions
\begin{equation*}
\pi(Z) \subset \widehat{\sigma_r}(X) \subset \widehat{\Sigma}_{\kappa \leq (r,2r, r)}
= \pi(Z).
\end{equation*}
The first inclusion is clear, the second is by
Proposition~\ref{prop:vanish:semi}, and the equality follows from
Lemma~\ref{lem:subspace-projection}. Therefore, these schemes must be equal,
which is the desired statement.
\end{proof}

We conclude by observing that Theorem~\ref{thm:main} is false for $r=7$. (We do
not know whether or not it holds for $r=6$.)

\begin{example}\label{ex:notequal}
Set $n = \dim V^*=6$, in which case $\Sigma_{\kappa\leq
(7,14,7)}=\Sigma_{\kappa_1\leq 14}$.  Let $X$ be the Segre-Veronese variety of
$\PP^2 \times \PP^5$ embedded by $\mathcal O(1,2)$ in $\PP^{62}$. We use a simple dimension count to show that the secant $\sigma_7(X)$ is properly contained
in $\Sigma_{\kappa_1\leq 14}$.

The secant variety $\sigma_7(X)$ is not defective~\cite[Corollary~1.4(ii)]{abo}, so it has the expected dimension, namely $\dim \sigma_7(X)= 7\cdot \dim
X+6=55$.  On the other hand, since $I_{\kappa_1\leq 14}$ is a Pfaffian ideal,
its codimension is at most $\binom{4}{2}$.  We thus have
\[
\dim \Sigma_{\kappa_1\leq 14}\geq \dim \mathbb P^{62}-\binom{4}{2}=62-6=56.
\]
Since $56>55$, it follows that $\sigma_7(X)\subsetneq \Sigma_{\kappa_1\leq 14}$.

Note that $\dim V^* = 6$ is the smallest dimension such that the $7$th secant
variety is properly contained within~$\PP^N$.

\end{example}

\bibliography{secant}
\bibliographystyle{plain}
\end{document}